\documentclass[oneside]{amsart}
\usepackage{epsfig,amsmath,latexsym,amssymb}
\usepackage{graphicx}
\usepackage{amsfonts}
\usepackage{amsthm}
\usepackage{color,bm}
\newtheorem{theorem}{Theorem}[section]
\newtheorem{lemma}[theorem]{Lemma}

\newtheorem{corollary}[theorem]{Corollary}

\def\T{\mathbb{T} } 
\def\S{\mathbb{S} }

\def\D{\mathbb{D} } 
 
\def\R{\mathbb{R} } 
\def\Z{\mathbb{Z} } 
\def\nbd{neighborhood } 
\def\nbds{neighborhoods } 
\def\R{\mathbb{R} } 
\def\iff{if and only if }

\def\-{\ominus} 
\def\+{\oplus} 
\def\0{\circ}

\usepackage{color,bm}

\title{Graph representations of surface flows}
\author{Tomoo Yokoyama}
\date{\today}
\address{Department of Mathematics, Kyoto University of Education/JST PRESTO, 
1 Fujinomori, Fukakusa, Fushimi-ku Kyoto, 612-8522, Japan \\
}
\email{tomoo@kyokyo-u.ac.jp}


\thanks{The author is partially supported
by the JST PRESTO Program at Department of Mathematics, Kyoto University of Education.}

\begin{document}

\begin{abstract}
We construct a complete invariant for non-wandering surface flows with finitely many singular points but without locally dense orbits. Precisely, we show that a flow $v$ with finitely many singular points on a compact connected surface $S$ is a non-wandering flow without locally dense orbits if and only if $S/v_{\mathrm{ex}}$ is a non-trivial embedded multi-graph, where the extended orbit space $S/v_{\mathrm{ex}}$ is the quotient space defined by $x \sim y$ if they belong to either a same orbit or a same multi-saddle connection. Moreover, collapsing edges of the non-trivial embedded multi-graph $S/v_{\mathrm{ex}}$ into singletons, the quotient space $(S/v_{\mathrm{ex}})/\sim_E$ is an abstract multi-graph with the Alexandroff topology with respect to the specialization order.  
Therefore the non-wandering flow $v$ with finitely many singular points but without locally dense orbits 
can be reconstruct by finite combinatorial structures, which are the multi-saddle connection diagram and the abstract multi-graph $(S/v_{\mathrm{ex}})/\sim_E$ with labels. 
Moreover, though the set of topological equivalent classes of irrational rotations (i.e. minimal flows) on a torus is uncountable,  
the set of topological equivalent classes of non-wandering flows with finitely many singular points but without locally dense orbits on compact surfaces is enumerable by combinatorial structures algorithmically.  

\end{abstract}

\maketitle

\section{Introduction}
The main purpose of this paper is to study the relationship 
between 
surface flows 
and 
topologies   
using finite graphs. 
A basic result of Morse theory   
says that 
gradient flows of Morse functions on closed surfaces are characterized by the set of separatrices of saddles, 
which are finite directed graphs. 
The Morse theory for gradient vector fields on compact manifolds is 
extended to an index theory for Smale flows on compact manifolds using Lyapunov graphs \cite{F}, 
which is a generalization of a quotient space of gradient functions. 
In \cite{RF}, 
a characterization of Lyapunov graphs associated to smooth surface flows is presented.
It's known that Lyapunov graphs 
are not complete invariant 
for Morse-Smale flows (i.e. 
there are Morse-Smale flows with isomorphic Lyapunov graphs but which are not topologically equivalent)  
but 
that  
Peixoto graphs are complete invariant for Morse-Smale flows \cite{N2}. 
In \cite{N}, 
non-wandering flows with finitely many singular points on compact surfaces are classified 
up to a graph-equivalence 
by using a  topological invariant, called a Conley-Lyapunov-Peixoto graph, 
equipped with the rotation and the weight functions. 
The graph-equivalence conjugates two non-wandering flows at the multi-saddle connection diagrams 
forgetting the equivalence between the quasi-minimal sets.
Moreover, 
Hamiltonian flows are topologically equivalent 
if and only if 
their Conley-Lyapunov-Peixoto graphs are isomorphic. 
%
On the other hand,   
the quotient maps of orbit spaces of generalized gradient vector fields for Morse functions
are weak homotopy equivalent and have the path lifting property \cite{CG}. 
In this paper, 
we study properties of surface flows
and 
construct another complete invariant, 
which are finite labeled multi-graphs, 
for a non-wandering surface flow  
with finitely many singular points but without locally dense orbits.  
In other words, 
such a flow
can be reconstruct by finite combinatorial structures. 
%
%
Precisely, we show the following statements. 
Let $v$ be a flow with 
finitely many singular points 
on a compact connected surface $S$. 
Then 
$v$ is a non-wandering flow with 
$\mathrm{LD} = \emptyset$ 
\iff  
the extended orbit space $S/v_{\mathrm{ex}}$ is a non-trivial embedded multi-graph, 
where 
$\mathrm{LD}$ is the union of locally dense orbits 
and 
$S/v_{\mathrm{ex}}$ 
is the quotient space defined by 
$x \sim y$ if they belong to 
either a same orbit 
or a same multi-saddle connection. 
Moreover, 
collapsing edges of the multi-graph $S/v_{\mathrm{ex}}$ into singletons,  
the quotient space $(S/v_{\mathrm{ex}})/\sim_E$ is 
an abstract multi-graph with the Alexandroff topology 
with respect to the specialization order.  
In addition, 
considering the multi-saddle connection diagram $D$ of 
a non-wandering flow $v$ with 
$|\mathop{\mathrm{Sing}}(v)| < \infty$ and 
$\mathrm{LD} = \emptyset$ on a compact surface, 
the flow $v$ can be reconstruct by 
the abstract 
multi-saddle connection diagram $D/v$, which is an abstract multi-graph,  
and 
by the abstract multi-graph $(S/v_{\mathrm{ex}})/\sim_E$ with labels, 
both of which are finite combinatorial structures.

Since the set of topological equivalent classes of 
irrational rotations (i.e. minimal flows) on a torus 
is uncountable (cf. Theorem 7.1.5 \cite{NZ}), 
so is the set of topological equivalent classes of 
non-wandering flows with finitely many singular points on compact surfaces. 
On the other hands, 
the set is countable under the non-existence of locally dense orbits. 
In other words, the set of topological equivalent classes of 
non-wandering flows with finitely many singular points but without locally dense orbits 
on compact surfaces is enumerable by combinatorial structures 
algorithmically.


\section{Preliminaries}

\subsection{Notions of dynamical systems}
We recall some basic notions. A good reference for most of what we describe are the book by S. Aranson, G. Belitsky, and E. Zhuzhoma \cite{ABZ}. 
By flows, 
we mean continuous $\mathbb{R}$-actions on surfaces. 
Let $v$ be a flow on a compact surface $S$. 
A subset of $S$ is said to be saturated 
if 
it is a union of orbits. 
The saturation $\mathrm{Sat}_v(A)$ of a subset $A \subseteq S$ is 
the union of orbits of elements of $A$. 
Recall that 
a point $x$ of $S$ is 
singular if 
$x = v_t(x)$ for any $t \in \R$, 
is regular if 
$x$ is not singular, 
and 
is periodic if 
there is positive number $T > 0$
such that 
$x = v_T(x)$ and  
$x \neq v_t(x)$ for any $t \in (0, T)$. 
Denote by 
$\mathop{\mathrm{Sing}}(v)$ 
(resp. $\mathop{\mathrm{Per}}(v)$) 
the set of singular (resp. periodic) points. 
A point $x$ is 
non-wandering if  
for each neighbourhood $U$ of $x$ and 
each positive number $N$, 
there is $t \in \mathbb{R}$ with $|t| > N$ such that 
$v_t(U) \cap U \neq \emptyset$. 
An orbit is non-wandering if 
it consists of non-wandering points 
and 
the flow $v$ is non-wandering if 
every point is non-wandering. 
For a point $x \in S$, 
define 
the omega limit set $\omega(x)$
and 
the alpha limit set $\alpha(x)$ 
of $x$
as follows: 
$\omega(x) 
:= \bigcap_{n\in \mathbb{R}}\overline{\{v_t(x) \mid t > n\}} 
$, 
$\alpha(x) 
:= \bigcap_{n\in \mathbb{R}}\overline{\{v_t(x) \mid t < n\}} 
$. 
A point $x$ of $S$ is 
recurrent (resp. weakly recurrent) 
if $x \in \omega(x) \cap \alpha(x)$ 
(resp.  $x \in \omega(x) \cup \alpha(x)$).  
%
A quasi-minimal set is an orbit closure of a weakly recurrent orbit. 
%
%
An orbit is proper if 
it is embedded (i.e. there is a \nbd of it where the orbit is closed), 
locally dense if 
the closure of it has nonempty interior, 
and 
exceptional if 
it is neither proper nor locally dense.  
A point is proper (resp. locally dense, exceptional) if 
so is its orbit. 
Denote by 
$\mathrm{LD}$ 
(resp. $\mathrm{E}$, 
$\mathrm{Pr}$, 
$\mathrm{P}$)
the union of locally dense orbits 
(resp. exceptional orbits, 
proper orbits, 
non-closed proper orbits).  
Denote by 
$\mathop{\mathrm{Cl}}(v)$ the union of closed orbits. 
In other words, 
we define 
$\mathop{\mathrm{Cl}}(v) := \mathop{\mathrm{Sing}}(v) \sqcup \mathop{\mathrm{Per}}(v)$, 
where $\sqcup$ is the disjoint union symbol. 
By the definitions,  
we have a decomposition 
$\mathop{\mathrm{Sing}}(v) \sqcup 
\mathop{\mathrm{Per}}(v) \sqcup \mathrm{P}
\sqcup \mathrm{LD} 
\sqcup \mathrm{E} = S$.  
Note that 
$\mathrm{P}$ is the complement of 
the set of weakly recurrent points.  
%
%
%
%
%
%
%
%
Recall that 
the (orbit) class $\hat{O}$ of an orbit $O$ is 
the union of orbits each of whose orbit closure 
corresponds with $\overline{O}$ 
(i.e. 
$\hat{O} = \{ y \in S \mid \overline{O(y)} = \overline{O} \} $). 
The quotient space by (orbit) classes
 is called 
the (orbit) class space 
and denoted by $S/\hat{v}$. 
%
A separatrix is a regular orbit whose $\alpha$-limit or $\omega$-limit set is a singular point. 
A $\partial$-$k$-saddle (resp. $k$-saddle) is 
an isolated singular point on (resp. outside of) $\partial S$ with exactly $(2k + 2)$-separatrices, 
counted with multiplicity, 
where $\partial S$ is the boundary of a surface $S$. 
For a ($\partial$-)$k$-saddle $x$, define the degree $\deg x := 2k + 2$.  
A $1$-saddle is topologically an ordinary saddle. 
A multi-saddle is 
a $k$-saddle 
or a $\partial$-$(k/2)$-saddle for some $k \in \mathbb{Z}_{\geq 0}$. 
The (multi-)saddle connection diagram 
is the union of (multi-)saddles, 
(multi-)$\partial$-saddles, and 
separatrices connecting (multi-)($\partial$-)saddles. 
A (multi-)saddle connection is 
a connected component of 
the (multi-)saddle connection diagram. 
A multi-saddle connection is also 
called a polycycle. 
%
By an extended orbit (resp. class) of a flow, 
we mean a multi-saddle connection or 
an orbit (resp. class) which is not contained in any multi-saddle connection.  
Note that 
an extended orbit is an analogous notion of  
``demi-caract\'eristique'' in the sense of Poincar\'e. 
In other words, 
``an extended positive orbit'' is ``un demi-caract\'eristique'' 
in the sense of Poincar\'e \cite{P}. 
An extended orbit is non-degenerate if 
it is not a singleton. 
Denote by $O_{\mathrm{ex}}(x)$ the extended orbit containing $x$ 
and 
by $\hat{O}_{\mathrm{ex}}(x)$ the extended class containing $x$. 
The quotient space by extended orbits (resp. classes) of a flow $v$ on a surface $S$ is called 
the extended orbit (resp. class) space 
and denoted by ${S}/{v_{\mathrm{ex}}}$ (resp. ${S}/\hat{v}_{\mathrm{ex}}$). 
Recall that 
an isolated singular point is called a topological center if 
there is a saturated \nbd consisting of closed orbits. 
An singular point $p$ is quasi-isolated if 
there is a saturated \nbd $U$ 
such that $\mathop{\mathrm{Sing}}(v) \cap U$ is totally disconnected. 
A quasi-isolated singular point $p$ is called a quasi-center if 
there is a saturated \nbd $U$ 
which consists of closed extended orbits. 
A flow is non-trivial if 
it is 
neither identical nor minimal. 
Note that 
the orbit spaces 
$S/v = S/v_{\mathrm{ex}}$ for each trivial flow $v$ on a surface $S$ are  
either the whole surface $S$ or 
a indiscrete topological space
 and 
 other quotient spaces 
$S/\hat{v} ={S}/{\hat{v}_{\mathrm{ex}}}$ 
are 
either the whole surface $S$ or a singleton 
and so that all the quotient spaces are non-$1$-dimensional.  
Moreover, 
notice that 
two non-wandering minimal flows 
need not be topological equivalent even if 
their (extended) orbit spaces (resp. (extended) orbit class spaces) are homeomorphic.  
In fact, consider 
minimal toral flows which correspond to cohomology classes 
$[dx_1 + \sqrt{2} dx_2],  
[dx_1 + \pi dx_2] \in H^1(\mathbb{T}^2, \Z)$ respectively.  
Though 
the orbit spaces are indiscrete spaces with cardinalities of the continuum 
and 
the (extended) orbit class spaces are singletons, 
the minimal flows are not topological equivalent 
because $\pi$ is transcendental and $\sqrt{2}$ is algebraic. 

A periodic orbit $O$ 
is called a periodic orbit with a one-sided neighborhood 
if 
there is an open saturated \nbd $U$ of $O$ 
and the canonical projection $\pi:U \to U/v \cong [0,1)$ 
with $\pi(O) = 0$ (i.e. $\pi(O)$ is the boundary of a half-open interval). 
Note that 
each periodic orbit of a non-wandering flow 
in $\partial S$ is a periodic orbit with a one-sided neighborhood. 
We call that 
a torus (resp. annulus) $U \subseteq S$ is a periodic torus (resp. periodic annulus) if 
it consists of periodic orbits.  

\subsection{Topological notions}

Recall that 
a poset is a set with a partial order 
(i.e.  reflexive, antisymmetric, and transitive order) 
and 
a $k$-chain is a totally ordered set with $k+1$ elements. 
For an element $x$ of a poset $P$,  
the height of $x$ 
is at least $k \in \Z$ 
if there is a $k$-chain whose maximal element is $x$. 
The height of $P$ is the superior of 
heights of elements. 
For any $k \in \Z_{\geq 0}$, 
denote by $P_k$ the set of elements each of whose height is $k$. 
%
Elements 
$a$ and $b$ of a poset are comparable if 
either $a \leq b$ or $a \geq b$. 
A poset $P$ is said to be connected 
if for each pair $a , b \in P$, 
there is a finite sequence $(a_i)_{i =0}^n$
of $P$ 
from  $a=a_0$ to $b = a_n$ 
such that 
$a_i$ and $a_{i+1}$ are comparable for each $i = 0, 1, 2, \dots ,n - 1$.
A subset $A \subset P$ is a downset (resp. upset) if 
$A = \bigcup_{x \in A} \mathop{\downarrow} x$ 
(resp. $A = \bigcup_{x \in A} \mathop{\uparrow} x$ ), 
where $\mathop{\downarrow} x := \{ y \in P \mid y \leq x \}$ 
and 
$\mathop{\uparrow} x := \{ y \in P \mid x \leq y \}$. 
The topology on a poset $(P, \leq )$ which consists of all upsets 
is called the Alexandroff topology of $P$ and 
denoted by $\mathcal{A}(\leq)$. 
Note that each downset  is closed with respect to the Alexandroff topology $\mathcal{A}(\leq)$. 
A poset $P$ is said to be multi-graph-like 
if the height of $P$ is at most one 
and  
$| \mathop{\downarrow} x | \leq 3$ 
for any element $x \in P$.  
For a multi-graph-like poset $P$, 
each element of $P_0$ is called a vertex 
and 
each element of $P_1$ is called an edge. 
%
%

An ordered triple $G := (V, E, r)$ 
is an abstract multi-graph if 
$V$ and $E$ are sets 
and 
$r : E \to \{ \{ x,y \} \mid x, y \in V \}$. 
Each element of $V$ (resp. $E$) is called a vertex (resp. an edge). 
Then 
an abstract multi-graph $G$ can be considered as 
a multi-graph-like poset $(P, \leq_G)$ 
with $V = P_{0}$ and $E = P_1$ 
as follows: 
$P = V \sqcup E$ 
and 
$x <_G e$ if 
$x \in r(e)$. 
Conversely, 
a multi-graph-like poset $P$ can be considered as an abstract multi-graph 
with $V = P_{0}$, 
$E = P_1$, 
and 
$r: P_1 \to \{ \{ x,y \} \mid x, y \in V \}$ defined by 
$r(e) := \mathop{\downarrow}e - \{ e \}$. 
A multi-graph-like poset is said to be tree-like if 
it is a tree as an abstract multi-graph. 
A multi-graph $G = (V, E)$ is 
a cell complex whose dimension is at most one 
and 
which is a geometric realization of an abstract multi-graph 
(i.e. a graph which is permitted to have multiple edges and loops). 
By abuse of terminology, we also refer to $G$ as the underlying set $V \sqcup \bigsqcup E$. 
%
Note that 
a multi-graph $G = (V, E)$ is $T_2$  
and 
the multi-edge set $E$ consists of open intervals with $\bigsqcup E = G - V$. 
A multi-graph $G = (V, E)$ in a surface $S$ is embedded if 
for any $e \in E$  there is an open neighborhood of $e$ in $S$ 
which intersect no other edges and no vertices, 
and for any $v \in V$ 
there is an open neighborhood of $v$ 
which intersects 
no vertices except $v$ and 
no edges not incident to $v$. 
%
A planar multi-graph is a multi-graph that can be embedded in a puncture disk. 
In other words, it can be drawn in such a way that no edges cross each other. 
Such a drawing is called a plane graph. 
%
For a multi-graph $G$, 
define an equivalence relation $\sim_e$ as follows: 
for any points $x, y \in G$, 
$x \sim_e y$ if 
they belong to a same edge. 
Note that 
the quotient space $G/\sim_e$ 
of a multi-graph $G$ is obtained by 
collapsing edges into singletons 
and that  
a finite poset of height at most one corresponds with 
a finite abstract multi-hyper-graph. 

A point $x$ of a topological space $X$ is 
$T_0$ if 
for any point $y \neq x \in X$, 
there is an open subset $U$ of $X$ such that 
$|\{x, y \} \cap U| =1$.  
%
The specialization order $\leq_{\tau}$ on a $T_0$-space $(X, \tau)$ 
is defined as follows: 
$ x \leq_{\tau} y $ if  $ x \in \overline{\{ y \}}$. 
Note that each closed subset of a $T_0$-space $(X, \tau)$ is 
a downset with respect to the specialization order $\leq_{\tau}$.

\subsection{Graph representations of surface flows}

For a non-wandering flow $v$ on a compact surface $S$, 
denote by $V_{\mathrm{ex}}$ 
the set of elements each of which is 
either 
a quasi-center, 
a multi-saddle connection, 
or a periodic orbit with a one-sided neighborhood. 
Notice that 
a quasi-center is a topological center 
if $\mathop{\mathrm{Sing}}(v)$ is finite. 
Let $\mathcal{D}$ be the set of connected components of 
the quotient space $D/v$ of 
the multi-saddle connection diagram $D$. 
Define 
a label $l_{V} : V_{\mathrm{ex}} \to \{ c, n, b \} \sqcup \mathcal{D}$ as follows: 
$l_{V}(O) := c$ (resp. $n, b$) if $O$ is a quasi-center  
(resp. a periodic orbit with one-sided neighborhood off $\partial S$, a periodic orbit on $\partial S$) 
and  
$l_{V}(O) := O/v \subseteq D/v$ if $O$ is a multi-saddle connection.  
Put $\widetilde{E_{\mathrm{ex}}}$ 
the set of connected components of $\mathop{\mathrm{Per}}(v) - N$, 
where $N$ is the union of periodic orbits with one-sided neighborhoods. 
Write $E_{\mathrm{ex}} := \widetilde{E_{\mathrm{ex}}}/v$. 
Note that 
each element of $\widetilde{E_{\mathrm{ex}}}$ is an open annulus 
unless it is a periodic torus. 
We define the edge mapping $r_{\mathrm{ex}}$ and 
the label mapping $l_{\mathrm{ex}}$ as follows: 
Define 
$r_{\mathrm{ex}} : E_{\mathrm{ex}} \to \{ \{ O_{\mathrm{ex}}, O'_{\mathrm{ex}} \} \mid 
O_{\mathrm{ex}}, O'_{\mathrm{ex}} \in V_{\mathrm{ex}} \} \sqcup \{ \emptyset \}$ 
by $r_{\mathrm{ex}}([T]) := \emptyset$ 
for a periodic torus $T$ and 
by $r_{\mathrm{ex}}([U]) := 
\{ O_{\mathrm{ex}} \in V_{\mathrm{ex}} \mid \partial_+ U \subseteq O_{\mathrm{ex}} 
\text{ or } 
\partial_- U \subseteq O_{\mathrm{ex}} \}$ 
for a  periodic annulus $U$, 
where $\partial_- U$ (resp. $\partial_+ U$) are 
the negative (resp. positive) connected components of $\partial U$ 
with respect to the flow direction such that 
$\partial U = \partial_- U \cup \partial_+ U$.   
An Edge $[U] \in E_{\mathrm{ex}}$ is called a loop if
$|r_{\mathrm{ex}}([U])|= 1$. 
Define also 
the label $l_{\mathrm{ex}} : E_{\mathrm{ex}} \to \{ ( \partial_- U/v, \partial_+ U/v ) \mid U \in E_{\mathrm{ex}}  \} 
\sqcup \{ \emptyset \}$ of $v$ 
by $l_{\mathrm{ex}}([T]) :=  \emptyset$ for a periodic torus $T$ 
and  
by $l_{\mathrm{ex}}([U]) :=  ( \partial_- U/v, \partial_+ U/v )$ for a periodic annulus $U$. 
Define also the reduced label $\widetilde{l_{\mathrm{ex}}} : 
E_{\mathrm{ex}} \to \{ \{ \partial_- U/v, \partial_+ U/v \} \mid U \in E_{\mathrm{ex}}  \} \sqcup \{ \emptyset \}$ of $v$ forgetting orders of pairs.  
%
We will show that,  
for an edge $U \in E_{\mathrm{ex}}$ of a non-wandering flow $v$ on a compact surface $S$ with 
$\mathrm{LD} = \emptyset$ and $|\mathop{\mathrm{Sing}}(v)| < \infty$,  
each boundary component (i.e. connected component of the boundary) 
 $\partial_{\sigma} U$ (for some $\sigma \in \{ -, + \}$)
is either an element of $V_{\mathrm{ex}}$ or 
a proper subset of a multi-saddle connection,  and 
the quotient space $\partial_{\sigma} U/v$ is a finite multi-graph-like poset.   
Define an equivalent relation $\sim_E$ on $S/v_{\mathrm{ex}}$ as follows: 
$O_{\mathrm{ex}} \sim_E O'_{\mathrm{ex}}$ 
if 
they are contained in a connected component of 
$\mathrm{LD} \sqcup (\mathop{\mathrm{Per}}(v) - N)$, 
where $N$ is the union of periodic orbits with one-sided neighborhoods. 
Denote by $G_v := ((S/v_{\mathrm{ex}})/\sim_E, l_{V},  l_{\mathrm{ex}})$ 
the quotient space $(S/v_{\mathrm{ex}})/\sim_E$ with the label $(l_{V}, l_{\mathrm{ex}})$.  
%
%
Recall that a multi-graph is trivial if 
it has exactly one vertex and no edges. 
%
Note that 
the quotient space $(S/v_{\mathrm{ex}})/\sim_E$ 
for a non-wandering flow $v$ on a compact connected surface $S$ is trivial 
if and only if 
either $v$ is minimal (i.e. $S = \T^2 = \mathrm{LD}$) or 
$v$ is a rational rotation on $\T^2$ $($i.e. $S = \T^2 = \mathop{\mathrm{Per}}(v) )$.

%
%

\section{Properties of the multi-saddle connection diagram}

%

We state the relation between multi-graphs and abstract multi-graphs. 

\begin{lemma}\label{lem01}
The quotient space of a multi-graph by $\sim_e$ 
is a multi-graph-like poset 
with respect to the specialization order. 
\end{lemma}

\begin{proof}
Let $Q$ be the quotient space $G/\sim_e$ 
of a multi-graph $G$ equipped with the specialization order, 
$p: G \to Q$ the canonical projection,  
and $V$ the set of vertices. 
Then 
the complement $Q - V$ 
consists of edges.  
Since each singleton of $G$ is closed,  
the height of each element of $V$ is $0$. 
Since the closure of each edge contains exactly one or two vertices, 
the height of each edge is $1$ 
and so $Q$ is a multi-graph-like poset. 
\end{proof}

Applying the previous lemma to multi-saddle connections, 
we obtain a following statement.

\begin{lemma}
Let $v$ be a flow on a compact surface $S$ 
and $D$ the multi-saddle connection diagram. 
Then 
the restriction $D/v$ of the orbit space of $v$ 
is a multi-graph-like poset 
with respect to the specialization order. 
%
%
%
\end{lemma}

\begin{proof}
Notice that 
the multi-saddle connection diagram $D$ is a multi-graph 
such that 
the equivalent class of an element $x \in D$ by $\sim_e$ 
is an orbit. 
Applying Lemma \ref{lem01}, 
the restriction $D/v = D/\sim_e$ is desired. 
\end{proof}

We call $D/v$ the abstract multi-saddle connection diagram of $v$. 
Let $V_D$ be the set of multi-saddles 
and $E_D$ the set of separatrices. 
Define a label $l_{De} : E_{D} \to V_D$ 
by $l_{De} (O) :=  \alpha (O)$, 
where $\alpha (O)$ is the alpha limit set of an orbit $O$. 
Define a cyclic relation $\sim_c$ on $E_D^n$ 
as follows:  
$(O_1, O_2, \dots , O_n) \sim_c 
(O_{i_1}, O_{i_2}, \dots , O_{i_n})$ 
if there is an integer $k = 0, 1, \dots n-1$ such that 
$j - i_j \equiv k \mod n$. 
Denote by $E_D^* := \bigsqcup_{n \in \Z_{\geq 0}} E_D^n/\sim_c$ 
Define a label $l_{Dv} : V_{D} \to E_D^*$ as follows: 
by $l_{Dv} (x) :=  [O_1, \dots, O_{\deg x}]$ 
if $\{ O_1, \dots, O_{\deg x} \}$ is the set of separatrices 
containing $x$ 
and the separatrices $O_1, \dots, O_{\deg x}$ are arranged in the counterclockwise direction 
around $x$. 
Denote by $D_v := (D/v, l_{Dv}, l_{De})$ 
the abstract multi-saddle connection diagram $D/v$ with the label $( l_{Dv}, l_{De})$.  
Then the multi-saddle connection diagram $D$ of $v$ (as a plane graph) 
can be reconstructed by the finite deta $D_v$. 
%

\section{Separation axioms of orbit (class) spaces}

We characterize 
the embedded property of multi-saddle connection diagrams as follows.

\begin{lemma}\label{lem03}
Let $v$ be a flow 
on a compact surface $S$ 
and  $D$ the multi-saddle connection diagram. 
The quotient topology on $D/v$ is the Alexandroff topology of the poset $D/v$ 
if and only if 
$D$ is embedded as a multi-graph.  
\end{lemma}

\begin{proof}
Let $Q := D/v$ be the poset with respect to the specialization order.
Suppose that 
the quotient topology on $Q$ is the Alexandroff topology of the poset $Q$. 
Then each edge of $Q$ is an open subset of $Q$  
and 
the upset $\mathop{\uparrow} w$ 
for any vertex $w$ of $Q$ 
is an open subset in $Q$  
which consists of $w$ and the edges incident to $w$. 
This implies that 
$D$ is embedded as a multi-graph. 
Conversely, 
suppose that $D$ is embedded as a multi-graph. 
Assume the quotient topology on $D/v$ is not the Alexandroff topology of the poset $D/v$. 
Note that 
each closed subset of $Q$ 
is a downset. 
Then there is a downset $A \subseteq Q$ such that 
$\overline{A} - A \neq \emptyset$. 
Since $D$ is embedded, 
each edge is open 
and so 
$\overline{A} - A$ consists of height $0$ elements. 
Therefore there is a vertex $w \in \overline{A} - A$. 
Since $\overline{A} - A \neq \emptyset$, 
any open neighborhood of $w$ 
intersects either vertices except $w$ or edges not incident to $w$,  
which contradicts to the definition of ``embedded''. 
\end{proof}

We characterize 
properness using quotient spaces.

\begin{lemma}\label{lem16}
Let $v$ be a non-trivial flow on a compact  surface $S$ and 
$p: S \to S/v$ the canonical projection. 
The set of $T_0$ points in $S/v$ 
is $p(\mathrm{Pr})$. 
\end{lemma}

\begin{proof}
Since a quasi-minimal set contains 
a continuum of nontrivially recurrent orbits each of  which is dense
in the quasi-minimal set (Theorem VI\cite{C}), 
the restriction 
$(\mathrm{E} \sqcup \mathrm{LD})/v = 
(\mathrm{E} \sqcup \mathrm{LD})/v_{\mathrm{ex}}$ consists of non-$T_0$ points. 
Since $\mathrm{Pr} =  \mathop{\mathrm{Cl}}(v) \sqcup \mathrm{P} 
= S -  (\mathrm{E} \sqcup \mathrm{LD})$, 
the set of $T_0$ points in $S/v$ is contained in $p(\mathrm{Pr})$. 
Proposition 2.2\cite{Y} implies that 
$\hat{O} = \overline{O} \setminus (\mathop{\mathrm{Sing}}(v) \sqcup \mathrm{P}) 
\subset \mathrm{LD} \sqcup \mathrm{E}$ for an orbit $O \subset \mathrm{LD} \sqcup \mathrm{E}$. 
Therefore $\hat{O} \subseteq \mathrm{Pr}$ for any orbit $O \subseteq \mathrm{Pr}$. 
Since the union of non-closed orbits in $\mathrm{Pr}$ is $\mathrm{P}$, 
we obtain 
$\hat{O} \subseteq \mathrm{P}$ for  an orbit $O \subseteq \mathrm{P}$. 
Therefore it suffices to show that 
either 
$\overline{O} \cap O' = \emptyset$ 
or 
$O \cap \overline{O'} = \emptyset$ 
for any non-closed orbits $O \neq O' \subset \mathrm{P}$. 
Indeed, 
since $O$ is proper, 
there is an open \nbd $U$ of $O$ with $O = \overline{O} \cap U$. 
Then the saturation $\mathrm{Sat}_v(U)$ is open 
with $O = \overline{O} \cap \mathrm{Sat}_v(U)$. 
Assume that $O' \subset \mathrm{Sat}_v(U)$. 
Then $O' \subset \mathrm{Sat}_v(U) \setminus \overline{O}$ 
and so 
$\overline{O} \cap O' = \emptyset$. 
Otherwise 
$O' \cap \mathrm{Sat}_v(U) = \emptyset$ 
and so 
$O \cap \overline{O'} = \emptyset$. 
\end{proof}

Note that 
there is a non-wandering flow $v$ 
on $\T^2$ with $\mathrm{P} \neq \emptyset$ 
whose multi-saddle connection diagram is dense 
and so is not embedded 
such that 
$\mathop{\mathrm{Sing}}(v)$ is countable (and so totally disconnected)  
and that 
$q(\mathrm{Pr})$ is not the set of $T_0$ points in $S/v_{\mathrm{ex}}$
(e.g. Example 2.10\cite{Y}), 
where $q: S \to  S/v_{\mathrm{ex}}$ is the canonical projection. 
Conversely, we have the following characterization if 
the set of singular points is finite.

\begin{lemma}
Let $v$ be a non-trivial flow on a compact  surface $S$ and 
$q: S \to  S/v_{\mathrm{ex}}$ the canonical projection. 
The set of $T_0$ points in $S/v_{\mathrm{ex}}$
is $q(\mathrm{Pr})$ if $|\mathop{\mathrm{Sing}}(v)| < \infty$. 
\end{lemma}

\begin{proof}
The finiteness of singular points implies that 
the multi-saddle connection diagram consists of 
finitely many multi-saddle connections each of which is closed. 
Then 
Lemma \ref{lem16} implies the assertion. 
\end{proof}

Recall that 
a flow is pointwise almost periodic if 
each orbit closure is minimal 
(i.e. $S/\hat{v}$ is $T_1$). 
By Lemma 2.2 \cite{Y2}, 
a flow $v$ on a compact surface is $R$-closed if and only if 
$S/\hat{v}$ is $T_2$. 
We characterize 
non-closed proper orbits using quotient spaces.

\begin{lemma}\label{lemma15-}
Let $v$ be a non-trivial flow on a compact connected surface $S$ and 
$p: S \to S/v$ (resp. $q: S \to  S/\hat{v}$) 
the canonical projection.   
The following are equivalent: 

1) 
$\mathrm{P} = \emptyset$. 

2) 
The set of $T_1$ points in $S/v$ 
is $p(S - \mathrm{LD})$  
$($i.e. each orbit is closed or locally dense$)$.

3) 
The set of $T_1$ points in $S/\hat{v}$ 
is $q(S - \mathrm{LD})$.  
\\
In any case, 
the flow $v$ is non-wandering 
such that   
$\mathop{\mathrm{Per}}(v)/v$ is a $1$-dimensional manifold 
and $S = \mathop{\mathrm{Cl}}(v) \sqcup \mathrm{LD}$. 
\end{lemma}

\begin{proof}
Write $Y := S - \mathrm{LD}$. 
Suppose that $\mathrm{P} = \emptyset$. 
By Lemma 2.3 \cite{Y}, 
there are no exceptional orbits 
and so 
$Y = S - \mathrm{LD} = \mathop{\mathrm{Cl}}(v) = \mathrm{Pr}$ 
is the union of closed orbits. 
This means that 
each orbit is closed or locally dense.  
Therefore the set of $T_1$ points in $S/v$ is $p(Y) = p(\mathrm{Pr})$. 
%
Suppose that 
the set of $T_1$ points in $S/v$ is $p(Y)$.  
This means that 
each orbit in $Y$ is closed. 
By the non-triviality of $v$, 
the flow $v$ is not minimal and so 
the set of $T_1$ points in $S/\hat{v}$ is $q(Y)$.  
Suppose that 
the set of $T_1$ points in $S/\hat{v}$ is $q(Y)$.  
This means that 
each orbit closure is minimal or quasi-minimal 
and so that 
$\mathrm{P} = \emptyset$. 
%
In any case, 
the flow $v$ is non-wandering. 
By Corollary 2.9 \cite{Y}, 
each connected component of $\mathop{\mathrm{Per}}(v)$ 
is either 
a connected component of $S$, 
an open annulus, 
or an open M\"obius band. 
Thus each connected component of 
$\mathop{\mathrm{Per}}(v)/v$ 
is either a circle or an interval and so 
$\mathop{\mathrm{Per}}(v)/v$ is a $1$-dimensional manifold. 
\end{proof}

In the case without locally dense orbits, 
the following statement holds.

\begin{corollary}\label{lemma15}
Let $v$ be a non-trivial flow on a compact connected surface $S$. 
The following are equivalent: 

1) 
$\mathrm{P} \sqcup \mathrm{LD} = \emptyset$. 

2) 
$S/v$ is $T_1$ $($i.e. each orbit is closed$)$.

3) 
$S/\hat{v}$ is $T_1$ $($i.e. $v$  is pointwise almost periodic$)$. 
\\
In any case, 
the flow $v$ is non-wandering 
with 
$S/v = S/\hat{v}$ and  
$(S- \mathop{\mathrm{Sing}}(v))/v$ is a $1$-dimensional manifold. 
\end{corollary}

Note that 
there is a non-wandering flow $v$ with $\mathop{\mathrm{P}} \neq \emptyset$ 
and $(S- \mathop{\mathrm{Sing}}(v))/v \cong \S^1$. 
Indeed, 
replacing an periodic orbit of an rational rotation on a torus $\T^2$ to 
a $0$-saddle with a homoclinic separatrix, 
the resulting flow $v$ on $\T^2$ is desired. 
%
%
By dimension, 
we mean 
the small inductive dimension. 
By Urysohn's theorem, 
the Lebesgue covering dimension, 
the small inductive dimension, 
and 
the large inductive dimension are corresponded in normal spaces. 
%
A compact metrizable space $X$ 
whose inductive dimension is $n > 0$ 
is an $n$-dimensional Cantor-manifold 
if 
the complement $X - L$ for any closed subset $L$ of $X$ 
whose inductive dimension is less than $n - 1$ 
is connected.  
It's known that 
a compact connected manifold is a Cantor-manifold \cite{HM,T}. 
We characterize the Hausdorff property of the orbit (class) spaces. 
The orientable case of the following result has stated in \cite{Y2} 
(see Theorem 6.6 in the paper).

\begin{lemma}
Let $v$ be a non-trivial flow on a compact connected surface $S$ 
whose orbit space $S/v$ is $T_1$. 
The following are equivalent: 

1) 
$\mathop{\mathrm{Sing}}(v)$ consists of at most two topological centers. 

2) 
$S/v$ is $T_2$. 

3) 
$S/\hat{v}$ is $T_2$ $($i.e. $v$  is $R$-closed$)$. 
\\
In any case, 
the Euler characteristic of $S$ is non-negative and 
the orbit space $S/v$ is 
either an closed interval or a circle. 
\end{lemma}

\begin{proof}
Let $v$ be a non-trivial flow on a compact connected surface $S$ 
whose orbit space $S/v$ is $T_1$.  
This means that $S = \mathop{\mathrm{Cl}}(v)$. 
Corollary \ref{lemma15} implies that 
the conditions 2) and 3) are equivalent. 
Since $v$ is non-trivial, 
there is a periodic orbit $O$. 
Let $C$ be the connected component of $\mathop{\mathrm{Per}}(v)$ 
which contains $O$. 
As above, 
Corollary 2.9 \cite{Y} implies 
$C$
is 
either 
the whole surface $S$, 
an open annulus,  
or an open M\"obius band. 
Then the restriction $C/v$ is an interval or a circle. 
Suppose that 
$\mathop{\mathrm{Sing}}(v)$ consists of at most two topological centers. 
Then the dimension of $\mathop{\mathrm{Sing}}(v)$ is zero 
and 
each connected component of the boundary $\partial C$
is a topological center, 
where $\partial C := \overline{C} - \mathrm{int} C$. 
Since $S$ is a $2$-dimensional Cantor-manifold, 
the complement $S - \mathop{\mathrm{Sing}}(v)$ is connected 
and so 
$S = C \sqcup \mathop{\mathrm{Sing}}(v)$. 
By the Poincar\'e-Hopf index formula, 
the Euler characteristic of $S$ is non-negative. 
Since the restriction $C/v$ is an interval or a circle, 
the orbit space $S/v$ is 
either an closed interval or 
a circle. 
Conversely, 
suppose that $S/v$ is $T_2$.  
Since the restriction $C/v$ is an interval or a circle,  
each connected component of the boundary $\partial C = \overline{C} - C$ is 
an closed orbit. 
Therefore the boundary 
$\partial C$ consists of at most two singular points. 
We show that each point in $\partial C$ is a center. 
Otherwise 
there is a non-isolated singular point $x \in \partial C$. 
Fix an open small \nbd $U$ of $x$ whose closure is homeomorphic to a closed ball 
such that $\overline{U} \cap \partial C = \{ x \}$. 
Since $x$ is not isolated, 
we obtain $(U \setminus \overline{C}) \cap \mathop{\mathrm{Sing}}(v) \neq \emptyset$. 
Then $V := U \setminus \overline{C} = U \setminus (C \sqcup \{ x \})$ is a nonempty open subset 
with $\overline{V} \cap \overline{C} = \{ x \}$ and $\overline{U} \subseteq \overline{V} \cup \overline{C}$. 
Put $W_1 := \overline{U} \cap \overline{C} = \overline{U} \cap (C \sqcup \{ x \})$ 
and $W_2 := \overline{V} \subseteq \overline{U}$. 
By construction, 
we obtain $\mathrm{int} W_{1} \neq \emptyset$ and $\mathrm{int} W_{2} \neq \emptyset$. 
Since $U - \{ x \} = (U \cap C) \sqcup V$ and $\overline{V} \cap \overline{C} = \{ x \}$, 
we have $\overline{U} = W_1 \cup W_2$ and $W_1 \cap W_2 = \{ x \}$. 
This implies that 
$\overline{U} - \{ x \} = (W_1 - \{ x \}) \sqcup (W_2 - \{ x \})$ is disconnected, 
which contradicts that $\overline{U}$ is a two-dimensional Cantor manifold. 
Thus 
%
$\partial C$ consists of 
at most two singular points which are topological centers 
and so $\overline{C} = S$. 
This means that 
$C = \mathop{\mathrm{Per}}(v)$ 
and so $\partial C = \partial \mathop{\mathrm{Per}}(v) = \mathop{\mathrm{Sing}}(v)$. 
%
\end{proof}

The non-triviality of a flow implies that 
the extended orbit class space is not a singleton.

\begin{lemma}\label{lem03-08}
Let $v$ be a flow on a compact connected surface $S$. 
Then $S/\hat{v}_{\mathrm{ex}}$ is a singleton if and only if $v$ is minimal. 
\end{lemma}

\begin{proof}
The minimality of $v$ implies that $S/\hat{v}_{\mathrm{ex}}$ is a singleton. 
Conversely, suppose that $S/\hat{v}_{\mathrm{ex}}$ is a singleton.  
The openness of  $\mathop{\mathrm{Per}}(v)$ implies $\mathop{\mathrm{Per}}(v) = \emptyset$. 
Since each extended orbit is dense,  
each singular point is a multi-saddle.  
Since multi-saddles are isolated, 
we have $|\mathop{\mathrm{Sing}}(v)| < \infty$. 
This means that each multi-saddle connection is closed 
and so the multi-saddle connection diagram is empty. 
Therefore $\mathrm{Pr} = \emptyset$ and so 
$S = \mathrm{LD}$. 
This implies $\T^2 = \mathrm{LD}$ and so $v$ is minimal. 
\end{proof}

\begin{lemma}\label{lem03-09}
Let $v$ be a non-trivial flow on a compact connected surface $S$. 
If $S/\hat{v}_{\mathrm{ex}}$ is $T_1$, 
then $\mathrm{LD} = \emptyset$. 
\end{lemma}

\begin{proof}
Assume that $\mathrm{LD} \neq \emptyset$. 
By Lemma \ref{lem03-08}, 
the extended orbit class space $S/\hat{v}_{\mathrm{ex}}$ is not 
a singleton. 
By Lemma 2.1 and 2.3 \cite{Y2}, 
we have $\emptyset \neq \overline{\mathrm{LD}} - \mathrm{LD} \subseteq 
\mathop{\mathrm{Sing}}(v) \sqcup \mathrm{P}$. 
By the finiteness of quasi-minimal sets, 
since $S/\hat{v}_{\mathrm{ex}}$ is $T_1$, 
the intersection $\overline{\mathrm{LD}} \cap (\mathop{\mathrm{Sing}}(v) \sqcup \mathrm{P})$ 
consists of dense extended orbits. 
In particular, 
each singular point in $\overline{\mathrm{LD}}$ is a multi-saddle. 
Moreover, 
there are countable many multi-saddles in $\overline{\mathrm{LD}}$ 
and so there is a non-isolated singular point in $\overline{\mathrm{LD}}$ 
which is not a multi-saddle, which contradicts to the non-existence. 
\end{proof}

Second we consider quotient spaces with respect to the extended orbits 
to encode non-wandering surface flows. 
For a subset $A$, write $\delta A := A - \mathrm{int} A$.  

\begin{lemma}\label{lem15a}
Let $v$ be a non-trivial flow on a compact connected surface $S$. 
The following are equivalent: 

1) 
$\mathrm{LD} = \emptyset$, 
the multi-saddle connection diagram contains $\mathrm{P}$, and 
each multi-saddle connection is closed.  

2) 
$S/v_{\mathrm{ex}}$  is $T_1$.

3) 
$S/\hat{v}_{\mathrm{ex}}$  is $T_1$. 

In any case, 
the flow $v$ is non-wandering with 
$S/\hat{v}_{\mathrm{ex}} = S/v_{\mathrm{ex}}$. 
\end{lemma}

\begin{proof}
By the definition of an orbit class, 
the condition $2)$ implies the condition $3)$. 
Suppose that 
$S/\hat{v}_{\mathrm{ex}}$ is $T_1$.  
This means that 
each proper orbit is 
either closed or  
contained in a closed multi-saddle connection.  
Lemma \ref{lem03-09} implies $\mathrm{LD} = \emptyset$. 
Suppose that 
$\mathrm{LD} = \emptyset$, 
the multi-saddle connection diagram contains $\mathrm{P}$, and 
each multi-saddle connection is closed.  
Then $\mathrm{int} \, \mathrm{P} = \emptyset$. 
Lemma 2.3 \cite{Y} implies 
$\mathrm{E} = \emptyset$ 
and so 
$S = \overline{\mathop{\mathrm{Cl}}(v)} = \mathop{\mathrm{Cl}}(v) \sqcup \delta \mathrm{P}$. 
This implies that 
$S/v_{\mathrm{ex}}$  is $T_1$ 
and that 
$v$ is non-wandering. 
%
\end{proof}

The closed condition of multi-saddle connections in the 
previous lemma is necessary. 
Indeed, 
consider an rotation $v_0$ with respect to an axis on $\S^2$. 
Write $\{p_N, p_S \} := \mathop{\mathrm{Sing}}(v_0)$.  
Fix a periodic orbit $O \subset \S^2$. 
Consider a sequence $(x_n)_{n \in \Z}$ of pairwise distinct points $x_n \in O$ converging to a point $x$ 
positively and negatively from different sides. 
Using a bump function and 
replacing the orbit $O$ into a union of singular points and proper orbits, 
the resulting flow $v$ is a non-wandering flow 
whose singular points are countable 
with 
$O = \mathrm{P} \sqcup (\mathop{\mathrm{Sing}}(v) - \{p_N, p_S \})$ 
such that $S/v_{\mathrm{ex}}$ is not $T_1$.  
Indeed, 
since each singular point in $O$ except $x$ is a $0$-saddle, 
the union $O - \{ x \}$ is an extended orbit of $v$ whose closure is $O$.

\begin{lemma}\label{lem3-10}
Let $v$ be a non-trivial flow on a compact connected surface $S$ 
such that $S/v_{\mathrm{ex}}$ is $T_1$. 
The following are equivalent: 

1) 
$\mathop{\mathrm{Sing}}(v)$ is totally disconnected.

2) 
$S/{v}_{\mathrm{ex}}$ is $T_2$.  

3) 
$S/\hat{v}_{\mathrm{ex}}$ is $T_2$.  

4)
%
Eether 
$v$ is a rational rotation on $\T^2$ 
or 
$(S/v_{\mathrm{ex}})/\sim_E$ is 
a non-trivial multi-graph 
$(V_{\mathrm{ex}}, E_{\mathrm{ex}}, r_{\mathrm{ex}})$.  

In any case, 
$S = \overline{\mathop{\mathrm{Per}}(v)}$ 
and 
the set  $\mathop{\mathrm{Sing}}(v)$ consists of 
multi-saddles 
and 
quasi-centers. 
\end{lemma}

\begin{proof}
Let $v$ be a non-trivial flow on a compact connected surface $S$ 
whose extended orbit space $S/{v}_{\mathrm{ex}}$ is $T_1$.  
By Lemma \ref{lem15a}, 
we have 
$\mathrm{P}$ is contained in the multi-saddle connection diagram  
and $\mathrm{LD} \sqcup \mathrm{E} = \emptyset$ 
(i.e. $S = \mathop{\mathrm{Cl}}(v) \sqcup \delta \mathrm{P}$). 
Note the compact connected surface $S$ is a $2$-dimensional Cantor-manifold. 
If $v$ is a rational rotation on $\T^2$  (i.e. a periodic torus), 
then the assertion holds. 
Thus we may assume that 
$v$ is not a rational rotation on $\T^2$. 
The non-triviality implies the existence of singular points. 
Lemma \ref{lem15a} implies that 
the conditions 2) and 3) are equivalent 
and that 
each quasi-isolated singular point is either 
a quasi-center or a multi-saddle.  
Trivially, 
the condition 4) implies 3). 
%
%
Let $M \subseteq \mathop{\mathrm{Sing}}(v)$ be 
the set of 
multi-saddles 
and $D$ the multi-saddle connection diagram. 
Since 
each multi-saddle is isolated, 
the complement 
$\mathop{\mathrm{Sing}}(v) - M = S - (\mathop{\mathrm{Per}}(v) \sqcup D)$ 
is closed. 
Then 
$\partial (\mathop{\mathrm{Sing}}(v) - M) = 
\partial (\mathop{\mathrm{Per}}(v) \sqcup D)
\subseteq \partial \mathop{\mathrm{Per}}(v)$.  
Moreover, the set $M$ is  countable. 
Indeed, 
by the definition of multi-saddles, 
each multi-saddle has a \nbd which contains no other saddles. 
Since $S$ is second countable, 
the set of multi-saddles 
can be enumerated 
and so is countable. 
Since $C \subset  \mathrm{int} (\mathop{\mathrm{Per}}(v) \sqcup C)$ 
for each multi-saddle connection $C$, 
we have $D \subseteq \mathrm{int} (\mathop{\mathrm{Per}}(v) \sqcup D)$ 
and so 
$\partial (\mathop{\mathrm{Per}}(v) \sqcup D) \cap D = \emptyset$. 
%
Fix a singular point $x \in \partial (\mathop{\mathrm{Per}}(v) \sqcup D)$. 
Let $(x_n)_{n \geq 0}$ be a convergence sequence in $\mathop{\mathrm{Per}}(v)$ 
to $x$ 
and $X_x := \{ \lim_{n \to \infty} y_n \mid (y_n) \text{ is a convergence sequence}, 
\, \, y_n \in \bigcup_{m \geq n} O(x_m) \}$. 
We claim that  
the limit $X_x$ is connected. 
Indeed, 
otherwise 
there are disjoint open  subsets 
$U, V$ such that 
$X_x \subseteq U \sqcup V$, 
$U \cap X_x \neq \emptyset$,  and 
$V \cap X_x \neq \emptyset$. 
Then $K := S - (U \sqcup V) \subset S - X_x$ is closed 
and so sequentially compact. 
Fix $y_n \in O(x_n) \cap K$. 
The sequence 
$(y_n)$ has a convergent subsequence $(y_{k_n})$. 
Then 
$\lim_{n \to \infty} y_{k_n} \in K \cap X_x$, 
which contradict to the definitions of $K$ and $X_x$. 
Note 
$x \in X_x \cap \mathop{\mathrm{Sing}}(v)$. 
We claim that  
$X_x \subseteq \partial (\mathop{\mathrm{Sing}}(v) - M)$. 
Indeed, 
for any periodic orbit $O$, 
there is a closed saturated \nbd $U_O$ of $O$ in $\mathop{\mathrm{Per}}(v)$ 
and so 
$O(x_n) \cap U_O = \emptyset$ for any large integer $n$. 
This implies 
$X_x \cap \mathop{\mathrm{Per}}(v) = \emptyset$ 
and so $X_x \subseteq 
\overline{\mathop{\mathrm{Per}}(v) \sqcup D} - \mathop{\mathrm{Per}}(v) 
= (\partial (\mathop{\mathrm{Per}}(v) \sqcup D) \sqcup \mathrm{int}(\mathop{\mathrm{Per}}(v) \sqcup D)) 
- \mathop{\mathrm{Per}}(v)  
\subseteq \partial (\mathop{\mathrm{Per}}(v) \sqcup D) \sqcup D$. 
Since $C \subset  \mathrm{int} (\mathop{\mathrm{Per}}(v) \sqcup C)$ 
for each multi-saddle connection $C$, 
there is a closed saturated \nbd $U_C$ of $C$ 
with $U_C - C \subseteq \mathop{\mathrm{Per}}(v)$ 
and so 
$O(x_n) \cap U_C = \emptyset$ for any large integer $n$. 
This implies $X_x \cap D = \emptyset$ 
and so $X_x \subseteq \partial (\mathop{\mathrm{Per}}(v) \sqcup D) 
= \partial (\mathop{\mathrm{Sing}}(v) - M)$. 
Suppose that 
$\mathop{\mathrm{Sing}}(v)$ is totally disconnected. 
Then 
$X_x$ is a singleton  
and $S = \overline{\mathop{\mathrm{Per}}(v)}$. 
This implies that 
each connected component of $\partial \mathop{\mathrm{Per}}(v)$ 
is contained in a closed extended orbit 
(i.e. either a multi-saddle connection or quasi-center). 
Assume that  $S/{v}_{\mathrm{ex}}$ is not $T_2$. 
Then 
there are distinct singular points $y \neq z$ 
such that 
any saturated \nbds of them intersect. 
Therefore there is a sequence $(O_n)$ of periodic orbits 
such that $y, z \in \overline{\bigcup_{n} O_n}$. 
This means that 
there are two convergence sequences $(y_n)$ and $(z_n)$ 
to $y$ and $z$ respectively such that  $y_n, z_n \in O_n$. 
Therefore $y \neq z \in X_z$, which contradicts that $X_z$ is a singleton. 
Since the complement of the union 
of  quasi-centers 
and 
of the multi-saddle connection diagram 
is $\mathop{\mathrm{Per}}(v)$, 
the quotient space $S/v_{\mathrm{ex}}$ is a non-trivial multi-graph 
each of whose vertices is either 
a quasi-center, 
a multi-saddle connection, 
or a periodic orbit with a one-sided neighborhood, 
and 
each of whose edges is a connected component of $\mathop{\mathrm{Per}}(v) \setminus N$ 
and is an open annulus, 
where $N$ is the union of periodic orbits with one-sided neighborhoods. 
Therefore $S/v_{\mathrm{ex}}$ is a non-trivial multi-graph $(V_{\mathrm{ex}}, E_{\mathrm{ex}}, r_{\mathrm{ex}})$.  
Finally it suffices to show that 
the condition 2) implies 1). 
%

Suppose that $S/{v}_{\mathrm{ex}}$ is $T_2$.  
Let $C \subseteq \partial (\mathop{\mathrm{Sing}}(v) - M) = 
\partial (\mathop{\mathrm{Per}}(v) \sqcup D)$ 
be a connected component of the boundary. 
The Hausdorff property implies that 
$C \subseteq \partial \mathop{\mathrm{Per}}(v)$ 
is contained in one extended orbit. 
Since $C \subseteq \mathop{\mathrm{Sing}}(v)$, 
the component $C$ is a singular point.
This means 
each connected component of 
$\partial (\mathop{\mathrm{Sing}}(v) - M)$ 
is a singleton. 
Since $M$ consists of isolated singular points, 
the boundary $\partial \mathop{\mathrm{Sing}}(v)$ 
is totally disconnected. 
Since the boundary $\partial \mathop{\mathrm{Sing}}(v) = 
\mathop{\mathrm{Sing}}(v) - \mathrm{int} \mathop{\mathrm{Sing}}(v)$ is compact metrizable, 
it is zero-dimensional. 
Hence 
$\mathrm{int} \mathop{\mathrm{Sing}}(v) = \emptyset$. 
Indeed, 
otherwise 
the dimension of $\mathop{\mathrm{Sing}}(v)$ 
is two.  
Then the complement $S - \partial \mathop{\mathrm{Sing}}(v)$ 
is disconnected because 
it has two connected components such that 
one contains a connected component of $\mathrm{int} \mathop{\mathrm{Sing}}(v)$ 
and the another contains one of $\mathop{\mathrm{Per}}(v)$. 
Since $S$ is a $2$-dimensional Cantor-manifold, 
the boundary $\partial \mathop{\mathrm{Sing}}(v)$ is at least 
one dimensional, which contradicts that 
$\partial \mathop{\mathrm{Sing}}(v)$ is zero-dimensional. 
Then  
$\mathop{\mathrm{Sing}}(v) 
= 
\partial \mathop{\mathrm{Sing}}(v)$ 
is totally disconnected 
and 
$S = \overline{\mathop{\mathrm{Per}}(v)}$.  
\end{proof}

\begingroup
\renewcommand{\arraystretch}{1.4}
\begin{table}[htb]
  \begin{center}
    \begin{tabular}{|l|c|c|} \hline
 &  $S/v$  & $S/v_{\mathrm{ex}}$  \\ \hline 
 $T_0$  & $\mathrm{LD} \sqcup  \mathrm{E} = \emptyset$ & $\mathrm{LD} \sqcup  \mathrm{E} = \emptyset$   \\ \hline 
 $T_1$  & $\mathrm{P} \sqcup  \mathrm{LD} = \emptyset$  & 
 $\mathrm{P} \subset \mathrm{D}$  \,   and   \,   $\mathrm{LD} = \emptyset$  \,   and  \\ 
& & $D$ consists of closed multi-saddle connections
 \\ \hline 
 $T_2$  & $S/v$ is $T_1$  \,  and   
 & $S/v_{\mathrm{ex}}$ is $T_1$ \,  and     \\ 
& $|\mathop{\mathrm{Sing}}(v)| \leq 2$  & $\mathop{\mathrm{Sing}}(v)$ is totally disconnected \\ \hline 
    \end{tabular}
  \end{center}
    \caption{For a non-wandering flow $v$ on a connected compact surface $S$ with the multi-saddle connection diagram $D$, 
necessary and sufficient conditions for 
separation axioms on the orbit space $S/v$ and the extended orbit space $S/v_{\mathrm{ex}}$ 
are compared. } 
\end{table}
\endgroup

Note that 
the condition that $\mathop{\mathrm{Sing}}(v)$ is totally disconnected 
is necessary in the previous lemma, 
because 
there is a real analytic non-wandering toral flow 
whose (extended) orbit space is not $T_2$ 
but $T_1$. 
Indeed, 
consider a real analytic non-wandering flow $v$ on $\T^2 = (\R/\Z)^2$ 
defined by $v_t(x,y) = (x + \sin (2 \pi y) t, y)$ 
with $\T^2 = \mathop{\mathrm{Cl}}(v)$ 
and $\mathop{\mathrm{Sing}}(v) = \R/\Z \times \{0,1/2 \}$ 
such that the orbit space $\T^2/v$ is not $T_2$ but $T_1$. 
%
Moreover, 
the multi-graph $(S/v_{\mathrm{ex}})/\sim_E$ need not be finite 
even if the extended orbit space is $T_2$, 
because 
there is a spherical flow $v$ 
whose extended orbit space is $T_2$  
whose multi-graph $(S/v_{\mathrm{ex}})/\sim_E$ has infinitely many edges.  
Indeed, 
consider a non-trivial rotation $v_0$ with respect to an axis on $\S^2$ 
and a sequence $(x_n)_{n \in \Z_{>0}}$ of points converging to a topological center 
such that 
$O_{v_0}(x_m) \neq O_{v_0}(x_n)$ for 
any $m \neq n \in \Z_{>0}$. 
Using a bump function and 
replacing the orbit of $x_n$ into a $0$-saddle with a homoclinic orbit 
(resp. a homoclinic saddle connection with a center disk)
the resulting flow $v$ is a non-wandering flow with 
infinitely many singular points such that 
the extended orbit space $\S^2/v_{\mathrm{ex}}$ of $v$ is a closed interval 
(resp. tree).
%
 %
%
%

\begin{lemma}\label{lem3-11}
Let $v$ be a flow on a compact surface $S$. 
The following are equivalent: 

1) 
$v$ is  non-wandering. 

2)  
$\mathrm{int}\mathrm{P} = \emptyset$. 

3)  
$(\mathrm{int}\mathrm{P}) \sqcup \mathrm{E} = \emptyset$ 
$($i.e. $S = \mathrm{Cl}(v) \sqcup \delta \mathrm{P} \sqcup \mathrm{LD} )$. 

4)  
$S = \overline{\mathop{\mathrm{Cl}}(v) \sqcup \mathrm{LD}}$. 
\end{lemma}

\begin{proof} 
Recall that 
$S = \mathop{\mathrm{Cl}}(v) \sqcup \mathrm{P} \sqcup \mathrm{LD} \sqcup \mathrm{E}$ 
and that 
the union $\mathrm{P}$ is the set of points which are not weakly recurrent.  
By taking a double covering of $M$ if necessary, 
we may assume that $v$ is transversally orientable. 
By the Ma\v \i er theorem \cite{M,Ma}, 
the closure $\overline{\mathrm{E}}$ is a finite union of closures of exceptional orbits 
and so is nowhere dense.  
By Lemma 2.3 \cite{Y}, 
the union $\mathrm{P} \sqcup \mathrm{E}$ is a \nbd of $\mathrm{E}$. 
Then 
$\mathrm{int}(\mathrm{P} \setminus \overline{\mathrm{E}}) \neq \emptyset$ 
if $\mathrm{E} \neq \emptyset$. 
We show that the condition 2) implies the condition 1). 
Suppose that $\mathrm{int}\mathrm{P} = \emptyset$. 
Then the closure of the set of weakly recurrent points is the whole surface 
(i.e. $S = \overline{\mathop{\mathrm{Cl}}(v) \sqcup \mathrm{LD} \sqcup \mathrm{E}}$) 
and so $v$ is  non-wandering. 
Trivially, the condition 3) (resp. 4)) implies the condition 2). 
Finally, 
we show that the condition 1) implies the conditions 3) and 4). 
Suppose that $v$ is non-wandering. 
By Theorem III.2.12, III.2.15 \cite{BS}, 
the set of weakly recurrence points is dense in $S$. 
The density of weakly recurrence points implies that $\mathrm{int}\mathrm{P} = \emptyset$ 
and so $\mathrm{E} = \emptyset$.  
This implies that 
$S = \overline{\mathop{\mathrm{Cl}}(v) \sqcup \mathrm{LD}}$. 
\end{proof}

We obtain another characterization of the Hausdorff separation property of 
the extended orbit space. 

\begin{lemma}
Let $v$ be a non-trivial flow on a compact connected surface $S$.  
The following are equivalent: 

1) 
$S = \overline{\mathop{\mathrm{Cl}}(v)}$ and each singular point is either a multi-saddle or a quasi-center. 

2) 
$S/{v}_{\mathrm{ex}}$ is $T_2$.  
%
\end{lemma}

\begin{proof}
Suppose that $S = \overline{\mathop{\mathrm{Cl}}(v)}$ and 
each singular point is either a multi-saddle or a quasi-center. 
Then $S = \overline{\mathop{\mathrm{Cl}}(v)} 
= \mathop{\mathrm{Sing}}(v) \cup \overline{\mathop{\mathrm{Per}}(v)}$. 
By Lemma 2.3 \cite{Y}, we obtain 
$S= 
\mathop{\mathrm{Cl}}(v) 
\sqcup \delta \mathrm{P}$. 
Since $v$ is non-wandering, 
Proposition 2.6 \cite{Y} implies that 
each non-closed proper orbit is a separatrix connecting singular points. 
Assume that 
there is an non-closed extended orbit $O_{\mathrm{ex}}$. 
By hypothesis, 
the intersection $\mathrm{Sing}(v) \cap O_{\mathrm{ex}}$ consists of 
infinitely many multi-saddles. 
Then there is a sequence of multi-saddles contained in $O_{\mathrm{ex}}$ 
converging to a non-isolated singular point $x$. 
By hypothesis, 
the singular point $x$ is a quasi-center. 
By the definition of quasi-centers, 
the extended orbit $O_{\mathrm{ex}}$ is closed, which contradicts to the assumption. 
Thus each extended orbit is closed. 
By Lemma \ref{lem15a}, 
the extended orbit space $S/v_{\mathrm{ex}}$ is $T_1$. 
By Lemma \ref{lem3-10}, 
the extended orbit space $S/{v}_{\mathrm{ex}}$ is $T_2$.  
Conversely, 
suppose that $S/{v}_{\mathrm{ex}}$ is $T_2$. 
Lemma \ref{lem15a} implies that 
$\mathrm{LD} = \emptyset$ and 
the union $\mathrm{P}$ is contained in the multi-saddle connection diagram. 
This means that 
$\mathrm{int}\mathrm{P} = \emptyset$. 
By Lemma \ref{lem3-11},
we have $S = \overline{\mathop{\mathrm{Cl}}(v)} = \mathop{\mathrm{Cl}}(v) \sqcup \delta \mathrm{P}$. 
Moreover, 
Lemma \ref{lem3-10} implies that $\mathop{\mathrm{Sing}}(v)$ is totally disconnected. 
Since each extended orbit is closed, 
each singular point is either a multi-saddle or a quasi-center. 
\end{proof}

Summarize the characterization of separation axioms for orbit spaces and orbit class spaces.

\begin{theorem}\label{th01}
Let $v$ be a non-trivial flow 
on a compact connected surface $S$ 
and $D$ the multi-saddle connection diagram. 
The following holds: 

1) 
$S/v$ is $T_0$ 
\iff
$S = \mathrm{Pr}$ 
$($i.e. $S = \mathop{\mathrm{Cl}}(v) \sqcup \mathrm{P} )$. 

2) 
$S/v$ is $T_1$ 
\iff
$S = \mathop{\mathrm{Cl}}(v)$ 
$($i.e. $S = \mathop{\mathrm{Sing}}(v) \sqcup \mathop{\mathrm{Per}}(v) )$. 

3) 
$S/v$ is $T_2$ 
\iff 
$S = \mathop{\mathrm{Cl}}(v)$ and $|\mathop{\mathrm{Sing}}(v)| \leq 2$. 

4) 
$S/v_{\mathrm{ex}}$  is $T_1$ 
\iff 
each multi-saddle connection is closed  
such that  $\mathrm{P} \subset D$ and $\mathrm{LD} = \emptyset$. 
%

5) 
$S/v_{\mathrm{ex}}$  is $T_2$ 
\iff 
$S = \overline{\mathop{\mathrm{Cl}}(v)}$ and each singular point is either a multi-saddle or a quasi-center. 
\end{theorem}

%
%
%
%
%
%

\section{Characterizations of non-wandering surface flows with finitely many singular points}

We show a following statement.

\begin{lemma}\label{lem3-4}
Let $v$ be a flow 
on a compact connected surface $S$. 
The following are equivalent: 

1) 
The flow $v$ is non-wandering with 
$|\mathop{\mathrm{Sing}}(v)| < \infty$. 

2) 
$(\mathop{\mathrm{Sing}}(v) \sqcup \mathrm{P})/v$ is a finite multi-graph-like poset. 

In any case, 
the union $\mathop{\mathrm{Sing}}(v) \sqcup \mathrm{P}$ 
consists of centers and 
of the multi-saddle connection diagram. 
\end{lemma}

\begin{proof}
Suppose that $v$ is non-wandering with 
$|\mathop{\mathrm{Sing}}(v)| < \infty$. 
Since $\mathop{\mathrm{Sing}}(v)$ is finite, 
Theorem 3\cite{CGL} implies that 
each singular point is 
either topological center or a multi-saddle. 
Proposition 2.6 \cite{Y} implies that 
$\mathrm{P}$ is contained in the multi-saddle connection diagram 
and consists of finitely many orbits.  
This implies that 
$\mathop{\mathrm{Sing}}(v) \sqcup \mathrm{P}$ 
consists of finitely many orbits and 
is the union of centers and 
the multi-saddle connection diagram. 
Therefore 
a quotient space  
$(\mathop{\mathrm{Sing}}(v) \sqcup \mathrm{P})/v$ is a finite multi-graph-like poset. 
Conversely, 
suppose that 
$(\mathop{\mathrm{Sing}}(v) \sqcup \mathrm{P})/v$ is a finite multi-graph-like poset. 
Then 
$\mathop{\mathrm{Sing}}(v) \sqcup \mathrm{P}$ 
consists of finitely many orbits. 
This means that  
$|\mathop{\mathrm{Sing}}(v)| < \infty$ 
and that $\mathrm{int}\mathrm{P}  = \emptyset$. 
By Lemma \ref{lem3-11}, we have that 
$v$ is non-wandering. 
\end{proof}

\begin{corollary}
Let $v$ be a non-wandering flow on a compact connected surface $S$ 
with $|\mathop{\mathrm{Sing}}(v)| < \infty$ 
and $N$ the union of periodic orbits with one-sided neighborhoods.  
Then each boundary component $C$ 
for a connected component of $\mathrm{LD}$ (resp. $\mathop{\mathrm{Per}}(v) - N$) 
is either an element of $V_{\mathrm{ex}}$ or 
a proper subset of a multi-saddle connection 
such that 
the quotient space $C/v$ is a finite multi-graph-like poset.   
\end{corollary}

\begin{proof} 
Let $C$ be a boundary component  
for a connected component of $\mathrm{LD}$ (resp. $\mathop{\mathrm{Per}}(v) - N$). 
Lemma 2.4 \cite{Y} implies $\mathop{\mathrm{Per}}(v)$ is open. 
By Lemma \ref{lem3-4}, 
the union $\mathop{\mathrm{Sing}}(v) \sqcup \mathrm{P}$ 
consists of centers and 
of the multi-saddle connection diagram. 
Since $|\mathop{\mathrm{Sing}}(v)| < \infty$, 
the multi-saddle connection diagram $D$ 
consists of finitely many closed multi-saddle connections. 
This means that 
the union $\mathop{\mathrm{Sing}}(v) \sqcup \mathrm{P}$ 
consists of finitely many closed extended orbits. 
By Lemma \ref{lem3-11}, 
we have $S = \mathop{\mathrm{Cl}}(v) \sqcup \delta \mathrm{P} \sqcup \mathrm{LD}$. 
Since $\overline{\mathrm{LD}}$ is the finite union of quasi-minimal sets, 
Lemma 2.1 \cite{Y} implies that 
$\overline{\mathrm{LD}} \cap \mathop{\mathrm{Per}}(v) = \emptyset$. 
By Lemma 2.3 \cite{Y}, we have 
$\overline{\mathop{\mathrm{Per}}(v)} \cap \mathrm{LD} = \emptyset$ 
and so 
$C \subseteq \partial (\mathop{\mathrm{Per}}(v) - N) \cup \partial \mathrm{LD} = 
\mathop{\mathrm{Sing}}_C(v) \sqcup D \sqcup N$,   
where $\mathrm{Sing}(v)_C$ is the set of centers. 
%
This implies that $C$ is either a center, a periodic orbit with a one-sided neighborhood, 
or a closed saturated subset of a multi-saddle connection. 
\end{proof}

We characterize 
non-wandering surface flow 
with finitely many singular points 
using extended orbit spaces.

\begin{theorem}\label{th11}
Let $v$ be a flow with $|\mathop{\mathrm{Sing}}(v)| < \infty$  
on a compact connected surface $S$.  
The following are equivalent: 

1) 
$v$ is non-wandering.  

2) 
$\mathrm{int}\mathrm{P}  = \emptyset$. 

3) 
$(S - \mathrm{LD})/v_{\mathrm{ex}}$ is an embedded multi-graph. 

4) 
$(S/v_{\mathrm{ex}})/\sim_E$ is a finite poset of height at most one. 

5) 
The multi-saddle connection diagram contains $\mathrm{P}$. 
\end{theorem}

\begin{proof}
By Lemma \ref{lem3-4}, 
the conditions $1)$ and $5)$ are equivalent. 
We may assume that $v$ is non-trivial.
%
Let $p: S \to (S/v_{\mathrm{ex}})/\sim_E$ be the canonical projection. 
As above, we may assume that 
$v$ is not a rational rotation on $\T^2$. 
Lemma \ref{lem3-11} implies that 
the conditions $1)$ and $2)$ are equivalent. 
Suppose that 
$(S - \mathrm{LD})/v_{\mathrm{ex}}$ is an embedded multi-graph. 
Then 
each non-closed proper orbit is contained in a closed multi-saddle connection diagram. 
Since $|\mathop{\mathrm{Sing}}(v)| < \infty$, 
we have $\mathrm{int}\mathrm{P}  = \emptyset$ 
and so
$v$ is non-wandering. 
Suppose that 
$(S/v_{\mathrm{ex}})/\sim_E$ is a finite poset of height at most one. 
If $\mathrm{int}\mathrm{P} \neq \emptyset$, 
then $p(\mathrm{P})$ contains uncountably many height one elements. 
Thus $\mathrm{int}\mathrm{P}  = \emptyset$. 
%
Conversely, 
suppose that 
$v$ is non-wandering. 
Then the union of proper orbit  is $\mathrm{Pr} 
= \mathop{\mathrm{Cl}}(v) \sqcup \mathrm{P} = S - \mathrm{LD}$. 
Since $\mathop{\mathrm{Sing}}(v)$ is finite, 
Theorem 3\cite{CGL} implies that 
each singular point is 
either topological center or a multi-saddle. 
Proposition 2.6 \cite{Y} implies that 
$\mathrm{P}$ is contained in the multi-saddle connection diagram 
and so that 
$\mathrm{Pr}/v_{\mathrm{ex}} = 
(\mathop{\mathrm{Cl}}(v) \sqcup \mathrm{P})/v_{\mathrm{ex}}$ 
is $T_1$. 
Since the complement of the union 
of topological centers 
and 
of the multi-saddle connection diagram 
is $\mathop{\mathrm{Per}}(v) \sqcup \mathrm{LD}$, 
the quotient space $\mathrm{Pr}/v_{\mathrm{ex}}$ is a multi-graph 
each of whose vertices is either 
a topological center, 
a multi-saddle connection, 
or a periodic orbit with a one-sided neighborhood, 
and 
each of whose edges is a connected component of 
$\mathop{\mathrm{Per}}(v) - N$
where $N$ is the union of periodic orbits with one-sided neighborhoods. 
Since $\mathop{\mathrm{Sing}}(v)$ is finite, 
$\mathrm{Pr}/v_{\mathrm{ex}}$ is embedded. 
By Theorem 2.5 and Corollary 2.9 \cite{Y}, 
the unions $\mathop{\mathrm{Per}}(v)$ and $\mathrm{LD}$ 
are open. 
Therefore 
the complement of the union of $\mathop{\mathrm{Sing}}(v)$ and 
of the multi-saddle connection diagram 
consists of finitely many connected components 
each of which is contained in $\mathop{\mathrm{Per}}(v)$ or $\mathrm{LD}$.  
Moreover the boundary $\partial \mathrm{LD} = \overline{\mathrm{LD}} - \mathrm{LD}$ 
is contained in the finite union of 
singular points and multi-saddle connections. 
Then $p(\mathrm{LD})$ consists of finitely many height one elements  
and so $(S/v_{\mathrm{ex}})/\sim_E$ is a finite poset of height at most one. 
\end{proof}

The finiteness in the previous theorem is necessary, 
because of 
the previous real analytic non-wandering toral flow with $\mathrm{LD} = \emptyset$ 
whose extended orbit space is not a multi-graph. 
Moreover, 
there is a  non-wandering flow $v$ on a compact connected surface $S$ with $\mathrm{LD} = \emptyset$ such that 
$S/v_{\mathrm{ex}}$ is not $T_1$ but 
$\mathop{\mathrm{Sing}}(v)$ is totally disconnected. 
Indeed, 
consider an rotation $v_0$ with respect to an axis on $\S^2$ 
and a sequence $(x_n)_{n \in \Z_{>0}}$ of points converging to a periodic point $x$ 
such that 
$O_{v_0}(x_m) \neq O_{v_0}(x_n)$ for 
any $m \neq n \in \Z_{>0}$. 
Using a bump function and 
replacing the orbit of $x_n$ (resp. $x$) into a $0$-saddle (resp. singular point) with a homoclinic orbit, 
the resulting flow $v$ is a non-wandering flow with 
infinitely many singular points such that 
the extended orbit $(O_v)_{\mathrm{ex}} (y)$ by $v$ of a point $y \notin O_{v_0}(x)$ 
is the orbit $O_{v_0}(y)$ of $y$ by $v_0$ 
but 
an (extended) orbit $O_{v_0}(x) - \{ x \}$ of $v$ is not closed. 
%
%
%
We characterize 
non-wandering surface flow with
finitely many singular points 
using quotient spaces.

\begin{theorem}\label{th12}
Let $v$ be a flow 
on a compact connected surface $S$. 
The following are equivalent: 

1) 
The flow $v$ is non-wandering with 
$|\mathop{\mathrm{Sing}}(v)| < \infty$. 

2) 
The topology of $(S/\hat{v}_{\mathrm{ex}})/\sim_E$ is 
Alexandroff 
with respect to the specialization order.  

3) 
$(S/\hat{v}_{\mathrm{ex}})/\sim_E$ is a finite connected poset of height  at most one.

4) 
$(\mathop{\mathrm{Sing}}(v) \sqcup \mathrm{P})/v$ is a finite multi-graph-like poset. 
\\
In any case, 
the quotient space $(S/\hat{v}_{\mathrm{ex}})/\sim_E$ is a multi-graph-like connected finite poset 
if $\mathrm{LD} = \emptyset$.  
\end{theorem}

\begin{proof}
Obviously, the condition 4) implies 1). 
By Theorem \ref{th11}, 
the conditions 1) and 4) are equivalent. 
Let $Q := (S/\hat{v}_{\mathrm{ex}})/\sim_E$ be the poset with respect to the specialization order 
and $p: S \to Q$ the canonical projection. 
Suppose that 
$v$ is non-wandering with 
$|\mathop{\mathrm{Sing}}(v)| < \infty$. 
Let $D$ be the multi-saddle connection diagram. 
By Proposition 2.6 \cite{Y}, 
the complement $S - (\mathop{\mathrm{Sing}}(v) \cup D) = \mathop{\mathrm{Per}}(v) \sqcup \mathrm{LD}$ 
has finitely many connected components. 
Since $|\mathop{\mathrm{Sing}}(v)| < \infty$, 
Corollary 2.9 \cite{Y} implies that 
$\mathop{\mathrm{Per}}(v)$ and $\mathrm{LD}$ are open. 
%
Since $\mathop{\mathrm{Sing}}(v)$ is finite and 
the complement $S - (\mathop{\mathrm{Sing}}(v) \cup D) = \mathop{\mathrm{Per}}(v) \sqcup \mathrm{LD}$ 
has finitely many connected components, 
the quotient space 
$Q$ 
is a finite connected poset of height  at most one. 
Moreover, 
it is a multi-graph-like finite poset if $\mathrm{LD} = \emptyset$. 
Since each connected component of $\mathop{\mathrm{Per}}(v)$ (resp. $\mathrm{LD}$) is open, 
the topology of 
$Q$ 
is Alexandroff with respect to the specialization order. 
Suppose that 
the topology of $Q$ 
is 
Alexandroff 
with respect to the specialization order. 
Then the restriction $Q_0$ to the set of height $0$ elements  
is a discrete topological space. 
Since $\mathop{\mathrm{Sing}}(v) \subseteq Q_0$ is compact, 
we have $|\mathop{\mathrm{Sing}}(v)| < \infty$. 
Assume that 
$v$ is not non-wandering (i.e. wandering). 
In other words,  
there is an open wandering domain $V \neq \emptyset \subseteq \mathrm{int} \mathrm{P}$ 
(i.e. 
there is $N \in \R$ such that $v_t(V) \cap V = \emptyset$ for any $t > N$). 
For each orbit $O \subset  \mathrm{Sat}_v(V)$, 
we have $O \cap V = \overline{O} \cap V$ 
and so 
$O \cap \mathrm{Sat}_v(V) = \overline{O} \cap \mathrm{Sat}_v(V)$. 
Then $U := \mathrm{Sat}_v(V) \setminus D \subseteq \mathrm{int} \mathrm{P}$ is an open saturated subset, 
where $D$ is the multi-saddle connection diagram. 
Fix an orbit $O_0 \subset U$.  
%
%
Write $W := \bigcup \{ \overline{O} \mid O \subset U - O_0 \}$. 
Then $O_0 \subset \overline{W} - W$. 
Since the topology of $(S/\hat{v}_{\mathrm{ex}})/\sim_E$ is Alexandroff, 
the image $p(W)$ is closed and so 
the inverse image $W$ is also closed, which contradicts to $O_0 \subset \overline{W} - W$. 
Thus $v$ is non-wandering. 
Suppose that 
$Q$ is a finite connected poset of height  at most one. 
Since $\mathop{\mathrm{Sing}}(v) \subseteq Q_0$, 
we have $|\mathop{\mathrm{Sing}}(v)| < \infty$. 
Assume that $\mathrm{int}\mathrm{P} \neq \emptyset$. 
Since multi-saddles are countable, 
there are uncountably many orbits of height one in $\mathrm{P}$, 
which contradicts to the finiteness of $Q$. 
Thus $\mathrm{int}\mathrm{P} = \emptyset$ 
and so the flow $v$ is non-wandering. 
\end{proof}

Considering a compact connected surface $S \subseteq \S^2$, 
Theorem \ref{th11} and \ref{th12} 
imply the following statement. 



\begin{corollary}
Let $v$ be a flow on a compact connected surface $S \subseteq \S^2$. 
The following are equivalent: 

1) 
The flow $v$ is non-wandering with 
$|\mathop{\mathrm{Sing}}(v)| < \infty$. 

2) 
The quotient space $(S/v_{\mathrm{ex}})/\sim_E$ 
is a tree-like finite poset of height one.   

3) 
The topology of $(S/{v}_{\mathrm{ex}})/\sim_E$ is 
Alexandroff with respect to the specialization order.  

4) 
$(S - \mathop{\mathrm{Per}}(v))/v$ is a finite multi-graph-like poset. 
\end{corollary}

%
%
%
%

The finiteness in the previous corollary is necessary, 
because of the example after Lemma \ref{lem15a}.  

\section{Graph representations of non-wandering surface flows}

Theorem \ref{th12} implies that 
a non-wandering flow with finitely many singular points 
but without locally dense orbits 
on a compact surface 
can be reconstructed by the extended orbit space with labels  
and the multi-saddle connection diagram $D$.  
%
%
To state precisely, 
we state one lemma and 
define some notations as follows:  
Denote by 
$\chi$
the set of non-wandering flows 
with 
$|\mathop{\mathrm{Sing}}(v)| < \infty$ and 
$\mathrm{LD} = \emptyset$ 
on a compact surface $S$. 
We summarize the properties of flows in $\chi$ as follows. 

\begin{lemma}
Let $D$ be the multi-saddle connection diagram of a flow $v \in \chi$ and 
let $\mathrm{Sing}(v)_C$ be the set of centers. 
Then the followings hold: 

1) 
$S = \mathrm{Sing}(v) \sqcup \mathrm{Per}(v) \sqcup \delta \mathrm{P}
= \mathrm{Sing}(v)_C \sqcup \mathrm{Per}(v) \sqcup D$. 

2) 
Both $S/v_{\mathrm{ex}}$ and $D$ are embedded multi-graphs. 

3) 
Both $(S/v_{\mathrm{ex}})/\sim_E$ and $D/v$ are finite multi-graph-like posets. 
%
%
%
\end{lemma}

Note that $G_v$ is an abstract multi-graph with labels. 
Let $\mathcal{G}$ be the set of finite abstract multi-graphs. 
For $H \in \mathcal{G}$, denote by  $\mathcal{C}_H$ 
the set of connected components of $H$. 
Define a set $\mathcal{P}$ of pairs of graphs in $\mathcal{G}$ with labels as follows: 
$(G, l_1, l_2, H, l_3, l_4) \in \mathcal{P}$ 
if 
$G, H \in \mathcal{G}$, 
$l_1: V(G) \to \{ c, b, n \} \sqcup \mathcal{C}_H$, 
$l_2: E(G) \to (V(G) \sqcup \mathcal{C}_H)^2$, 
$l_3: V(H) \to E(H)^*$, 
and $l_4: E(H) \to V(H)$, 
where $E(H)^* := \bigsqcup_{n \in \Z_{\geq 0}} E(H)^n/\sim_c$ 
and $\sim_c$ is a cyclic relation on $E(H)^n$.  
Two pairs 
$(G, l_1, l_2, H, l_3, l_4), 
(G', l'_1, l'_2, H', l'_3, l'_4) \in \mathcal{P}$ 
are isomorphic if 
there are two graph isomorphisms $g: G \to G'$ and $h: H \to H'$
such that 
$\widetilde{h} \circ l_1 = l'_1 \circ g$,  
$(h \sqcup g)^2 \circ l_1 = l'_1 \circ g$,  
$h \circ l_i = l'_i \circ h$ ($i = 3, 4$), 
where $\widetilde{h}: \{ c, b, n \} \sqcup \mathcal{C}_H \to \{ c, b, n \} \sqcup \mathcal{C}_G$ 
is the extension of $h$ whose restriction on $\{c, b, n\}$ 
is identical,  
and $(h \sqcup g)^2: (V(G) \sqcup \mathcal{C}_H)^2 \to (V(G') \sqcup \mathcal{C}_{H'})^2$ 
is the induced mapping by $g$ and $h$. 
Denote by $\widetilde{\mathcal{P}}$ the quotient set of $\mathcal{P}$ by the isomorphisms. 

Define a mapping 
$p: \chi \to \mathcal{P}$ by 
$p(v) := ( G_v, D_v) = (((S/v_{\mathrm{ex}})/\sim_E, l_{V}, l_{\mathrm{ex}}), (D/v, l_{Dv}, l_{De}))$ 
and 
an equivalent relation $\sim_+$ on $\chi$ 
by $v \sim_+ w$ if 
$v$ and $w$ is topologically equivalent 
(i.e. there is a homeomorphism $h: S \to S$ 
such that  
the image of an orbit of $v$ is an orbit of $w$ 
and that it preserves orientation of the orbits).  
Moreover define 
an equivalent relation $\sim$ on $\chi$ 
by $v \sim w$ if 
either  $v \sim_+ w$ or 
there is a homeomorphism $h: S \to S$ 
such that the image of an orbit of $v$ is an orbit of $w$ 
and that it reverses orientation of the orbits.  
Write $\chi_{\sim}^+$ (resp. $\chi_{\sim}$) by the quotient space 
of $\chi$ by the topological equivalence $\sim_+$ (resp. the equivalence $\sim$). 
Now we can state precisely that 
a non-wandering flow with finitely many singular points 
but without locally dense orbits 
on a compact surface 
can be reconstructed by the abstract multi-graph with labels 
and the abstract multi-saddle connection diagram with labels.  
On the other words, 
we obtain the following statement.

\begin{theorem}\label{th34}
The induced map $\widetilde{p}: \chi_{\sim}^+ \to \widetilde{\mathcal{P}}$ 
is well-defined and injective. 
\end{theorem}

\begin{proof}
Since $S$ has only finitely many connected components, 
we may assume that $S$ is connected. 
Fix an element $v \in \chi$. 
We will show that  we can reconstruct the topologically equivalent class $[v]$ from 
$p(v) = ( G_v, D_v) = (((S/v_{\mathrm{ex}})/\sim_E, l_{V}, l_{\mathrm{ex}}), (D/v, l_{Dv}, l_{De}))$.  
Let $r: S \to (S/v_{\mathrm{ex}})/\sim_E$ be the canonical projection 
and $Q := (S/v_{\mathrm{ex}})/\sim_E$. 
By Theorem \ref{th12}, 
the quotient space $Q$ is a multi-graph-like connected finite poset. 
Since an isolated point $x$ in $Q$ corresponds to a periodic torus, 
we may assume that $Q$ has no isolated points. 
From the finite data $D_v$, 
we can reconstruct the multi-saddle connection diagram $D$ as a directed plane graph 
up to topological equivalence. 
Let $\mathcal{D}$ be the set of connected components of 
the quotient space $D/v$ of the multi-saddle connection diagram $D$. 
Since the set $Q_0$ of height $0$ elements in $Q$ 
corresponds to $V_{\mathrm{ex}}$, 
by the labels in $\mathrm{Im} (l_{V}) = \{ c, n, b \} \sqcup \mathcal{D}$, 
we obtain centers 
(resp. periodic orbits with one-sided neighborhoods off $\partial S$, 
periodic orbits on $\partial S$, multi-saddle connections) 
as directed plane graphs from both $Q_0$ and $D$ up to topological equivalence. 
Using the labels in  $\mathrm{Im} (l_{\mathrm{ex}}) = 
\{ ( \partial_- U/v, \partial_+ U/v ) \mid U \in \widetilde{E_{\mathrm{ex}}}  \}$, 
since any element in $\mathrm{Im} (l_{\mathrm{ex}})$ is 
an ordered pair of two elements each of which is contained in $V_{\mathrm{ex}}$, 
paste periodic annuli, which correspond to the set $Q_1$ of height $1$ elements in $Q$, 
between all pairs in $\mathrm{Im} (l_{\mathrm{ex}})$. 
Thus we reconstruct the flow $v$ uniquely up to topological equivalence. 
\end{proof}

If $S$ is closed and orientable, then the label $l_{V}$ is not necessary in the previous theorem.  
The previous theorem implies the following statement. 

\begin{corollary}\label{cor35}
The set 
of topological equivalent classes of 
non-wandering flows with finitely many singular points but without locally dense orbits 
on compact surfaces is enumerable by combinatorial structures 
algorithmically. 
\end{corollary}

\begin{proof}
Since both $(S/v_{\mathrm{ex}})/\sim_E$ and $D/v$ for an element $v \in \chi$  
are finite topological spaces, 
the labels $l_{V}, l_{\mathrm{ex}}, l_{Dv}$, and $l_{De}$ 
are mapping between finite sets. 
Since the set of compact surfaces is enumerable 
by induction on numbers of boundaries, genus, and connected components 
and since the set of multi-saddle connection diagrams is 
enumerable by induction on numbers of vertices and of separatrices,  
the assertion holds. 
\end{proof}

Let $\mathrm{TOP}$ be the set of topological space 
and $\mathrm{TOP}_{\sim}$ the quotient space by homeomorpisms. 
This implies the following corollary.

\begin{corollary}\label{cor36}
The projection $\pi: \chi_{\sim} \to \mathrm{TOP}_{\sim}$ 
defined by $\pi([v]) := S/v$ 
is well-defined 
such that 
both $D/v$ and $\widetilde{l_{\mathrm{ex}}}$ can be constructed by 
the orbit space $S/v$. 
\end{corollary}

\begin{proof}
Fix a flow $v \in \chi$. 
Let $q_0 : S \to S/v$ be the canonical projection.
We may assume that $S$ is connected. 
If $S/v$ is a circle, then $v$ is a rational rotation. 
Thus we may assume that $v$ is not a rational rotation. 
Let $Z \subset S/v$ be the set of non-$T_0$ points. 
Then $q_0^{-1}(Z) = D \setminus \mathop{\mathrm{Sing}}(v)$ 
and so $q_0^{-1}(\overline{Z)} = D$. 
Therefore $\overline{Z}$ is a finite multi-graph-like poset 
such that 
the saddle connection diagram $D$ is a realization of the abstract multi-graph $\overline{Z}$. 
An element $[x] \in S/v - Z$ is called a boundary point 
if there is a \nbd $W$ of $[x]$ 
such that $(W, [x])$ is homeomorphic of $([0,1), 0)$. 
Collapsing each multi-saddle connection into a singleton, 
we obtain a canonical projection $q_1 : S/v \to S/v_{\mathrm{ex}}$. 
Let $B \subset S/v$ be the set of boundary points. 
Then define 
$V_{\mathrm{ex}} := q_1(B \sqcup \overline{Z})$.  
Then the inverse image $q_0^{-1}(B)$ 
is the set of centers and periodic orbits with one-sided neighborhoods. 
Denote by $E_{\mathrm{ex}}$ 
the set of connected component of 
$S/v_{\mathrm{ex}} - q_1(B \sqcup\overline{Z}) 
=  S/v - (B \sqcup \overline{Z})$. 
For an element $[U] \in E_{\mathrm{ex}}$, 
the inverse image $U := (q_1 \circ q_0)^{-1}([U])$ is an open annulus. 
%
%
Therefore 
the boundary $q_0(\partial U) \subset S/v$ consists of one or two elements 
$q_0(\partial_- U), q_0(\partial_+ U) \subseteq (B \sqcup \overline{Z})/v 
=  q_0 (\mathop{\mathrm{Sing}}(v) \sqcup \mathrm{P})$. 
Since $[U]$ is an open interval, 
there are exactly two boundary components $\partial_- [U]$ and $\partial_+ [U]$ of $[U]$ in $S/v$. 
Note the image $q_1(\partial_- [U])$ (resp. $q_1(\partial_+ [U])$) of any $[U] \in E_{\mathrm{ex}}$ is a singleton. 
Then we can define  
a label $\widetilde{l_{\mathrm{ex}}} : E_{\mathrm{ex}} \to \{ \{ \partial_- [U], \partial_+ [U] \} \mid [U] \in E_{\mathrm{ex}}  \}$ of $v$ 
by $\widetilde{l_{\mathrm{ex}}}([U]) :=  \{ \partial_- [U], \partial_+ [U] \}$. 
\end{proof}

Note the labels $l_{De}$, $l_{V}$, $l_{Dv}$ can't be reconstructed by the orbit space $S/v$ 
in general. 
Indeed, 
since the orbit space does not know the flow direction,  
the label $l_{De}$ can't be reconstructed by the orbit space $S/v$. 
%
Since the orbit spaces of center disks, 
of a periodic M\"obius band, 
and of an closed annulus are closed interval,  
the three spaces can't be distinguished by 
their orbit spaces 
and so 
the label $l_{V}$ need not be reconstructed. 
Moreover, 
the label $l_{Dv}$ can't be reconstructed 
(see Fig. \ref{fig05}). 
Recall that 
a continuous flow is regular if each singular point has a \nbd 
which is topologically equivalent to a \nbd of a non-degenerate singular point. 
Note that 
each non-wandering flow on a closed surface with finitely many singular points is regular 
if and only if 
each singular point is either a center or a saddle. 
Let $\chi_r \subset \chi$ be the subset of regular flows in $\chi$ 
and $\chi_{r\sim} \subset \chi_{\sim}$ the quotient space by the equivalence $\sim$. 
Then each regular non-wandering flow 
with finitely many singular points but without locally dense orbits 
can be reconstructed by the orbit spaces as follows. 

\begin{corollary}\label{cor36a}
Suppose that $S$ is an orientable closed surface. 
The projection $\pi: \chi_{r\sim} \to \mathrm{TOP}_{\sim}$ 
defined by $\pi([v]) := S/v$ is injective. 
\end{corollary}

\begin{proof}
By Corollary \ref{cor36}, 
we can construct $D/v$ and $\widetilde{l_{\mathrm{ex}}}$. 
We may assume that $S$ is connected. 
Note that $S/v$ is a circle if and only if $S$ is a periodic torus. 
Thus we may assume that $S$ is not a periodic torus. 
Since $S$ is orientable and closed, the image of the label $l_{V}$ is $\{ c \}  \sqcup \mathcal{C}_D$, 
where $\mathcal{C}_D$ is the set of connected components of 
the multi-saddle connection diagram $D$.  
The regularity implies that each saddle connection consists of 
one saddle and two homoclinic separatrices 
and so the label $l_{Dv}$ can be constructed by $l_{V}$. 
By the regularity, the label $l_{De}$ maps a separatrix into the saddle in 
the homoclinic saddle connection containing it. 
Since the equivalence relation $\sim$ ignores the orientation of orbits, 
we don't need the orbit direction of $D$ to reconstruct the equivalent classes of $\chi_{r\sim}$. 
Therefore Theorem \ref{th34} implies the assertion. 
\end{proof}

Conversely, 
we consider the following problem: 
For a given abstract multi-graph, 
can it be realized 
by a non-wandering surface flow and in how many ways? 
Note that  
the abstract multi-graph $(\S^2/v_{\mathrm{ex}})/\sim_E$ 
by any non-wandering flow is a tree 
and that 
graphs can be realized by several graphs. 
We state realizability of graphs.

\begin{theorem}
For any non-trivial connected  finite abstract multi-graph $G$, 
there is a non-wandering flow $v$ on a closed surface $S$ 
such that $G$ is isomorphic to $(S/v_{\mathrm{ex}})/\sim_E$. 
\end{theorem}

\begin{proof}
Since any non-trivial finite multi-graph $G = (V, E)$ can be 
decomposed into height one trees $T_i = (V_i, E_i)$ by cutting all edges
into pairs of edges, 
such a multi-graph can be reconstructed by 
height one trees gluing the new height zero elements. 
Note 
$|V| + 2 |E| = \sum_{i}|V_i|$ 
and 
$2 |E| = \sum_{i}|E_i|$. 
Therefore it suffices to show that,  for any $k \in \mathbb{Z}_{> 0}$, 
there is a non-wandering flow $w$ on a compact surface $T$ such that 
$(T/w_{\mathrm{ex}})/\sim_E$ is isomorphic to $T_k$, 
where $T_k$ is a tree with $k$ points of height $0$ and $(k-1)$ points of height one.  
Let $S_1$ be a closed center disk of a flow.  
Consider a flow $w$ on the metric completion $S_{k+1}$ of a $k$ punctured disk 
which consists of one homoclinic $(k-1)$-saddle connection and of periodic orbits 
such that the boundaries are periodic orbits (e.g. Fig. \ref{fig01}). 
Then the quotient space $(S_k/w_{\mathrm{ex}})/\sim_E$ is isomorphic to $T_k$. 
Gluing boundaries of such surfaces 
and 
pasting center disks to periodic orbits on the boundaries, 
the resulting flow on the resulting closed surface is desired. 
\end{proof}

\begin{figure}
\begin{center}
\includegraphics[scale=0.2]{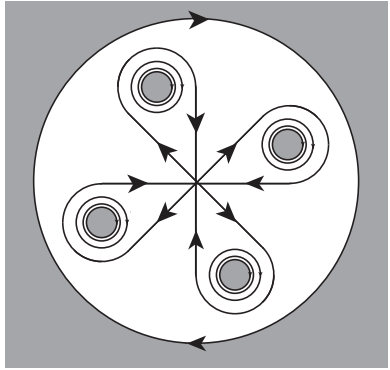}
\end{center} 
\caption{A homoclinic $3$-saddle connection and periodic orbits on $S_5$}\label{fig00}
\end{figure}

\section{Examples}

The finiteness in Theorem \ref{th34} 
(resp. Corollary \ref{cor36a}) is necessary. 
Indeed, 
there are two non-wandering flows 
with infinitely many singular points 
on $\mathbb{T}^2$ 
which are not topologically equivalent 
but which have 
the same (extended) orbit space up to homeomorphism 
and the same multi-saddle connection diagram 
(see Fig.\ref{fig01}). 
Moreover 
the non-wandering property in Theorem \ref{th34} (resp. Corollary \ref{cor36a}) 
is necessary. Indeed, 
there are two flows with wandering domains 
on $\mathbb{T}^2$ 
which are not topologically equivalent 
but which have 
the same (extended) orbit space  up to homeomorphism 
and the same multi-saddle connection diagram 
(see Fig.\ref{fig02}). 
As a same argument, 
%
the closedness (resp. orientability) is necessary in Corollary \ref{cor36a}. 
Indeed, consider an rotation $v_0$ with respect to an axis on $\S^2$. 
Replacing a periodic orbit into a saddle connection, 
we obtain a flow $v_1$ as in Fig.\ref{fig03}. 
Replacing a center disk $B_{\sigma}$ ($\sigma = x, y$) 
into a periodic orbit which is a boundary component (resp. a periodic M\"obius band)  
we can obtain two flows $v_{\sigma}$. 
Since the flow directions of a center disk containing $x$ 
is opposite to the one of a center disk containing $y$ (resp. $z$), 
the resulting flows $v_x$ and $v_y$ 
have the same extended orbit space  up to homeomorphism 
but are not topologically equivalent. 
%
The regularity is necessary in Corollary \ref{cor36a}. 
Indeed, there are 
two non-regular flows $v, w$ on a disk $\D^2$ 
which are not topologically equivalent such that  
the multi-saddle connection diagrams are not isomorphic 
as plane graphs (resp. abstract labeled multi-graphs) but isomorphic  
as abstract multi-graphs (see Fig.\ref{fig05}).

\begin{figure}
\begin{center}
\includegraphics[scale=0.475]{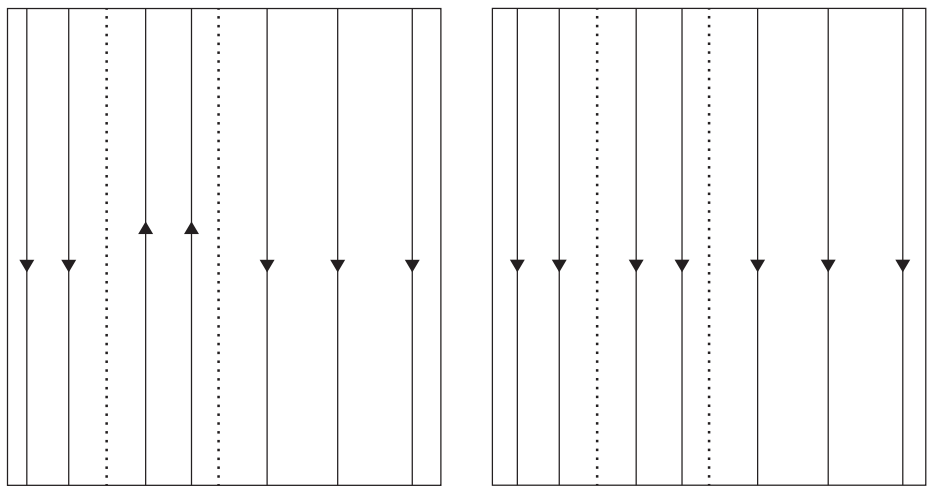}
\end{center} 
\caption{Two non-wandering flows 
which are not topologically equivalent 
but which have 
the same extended orbit space.}\label{fig01}
\end{figure}

\begin{figure}
\begin{center}
\includegraphics[scale=0.475]{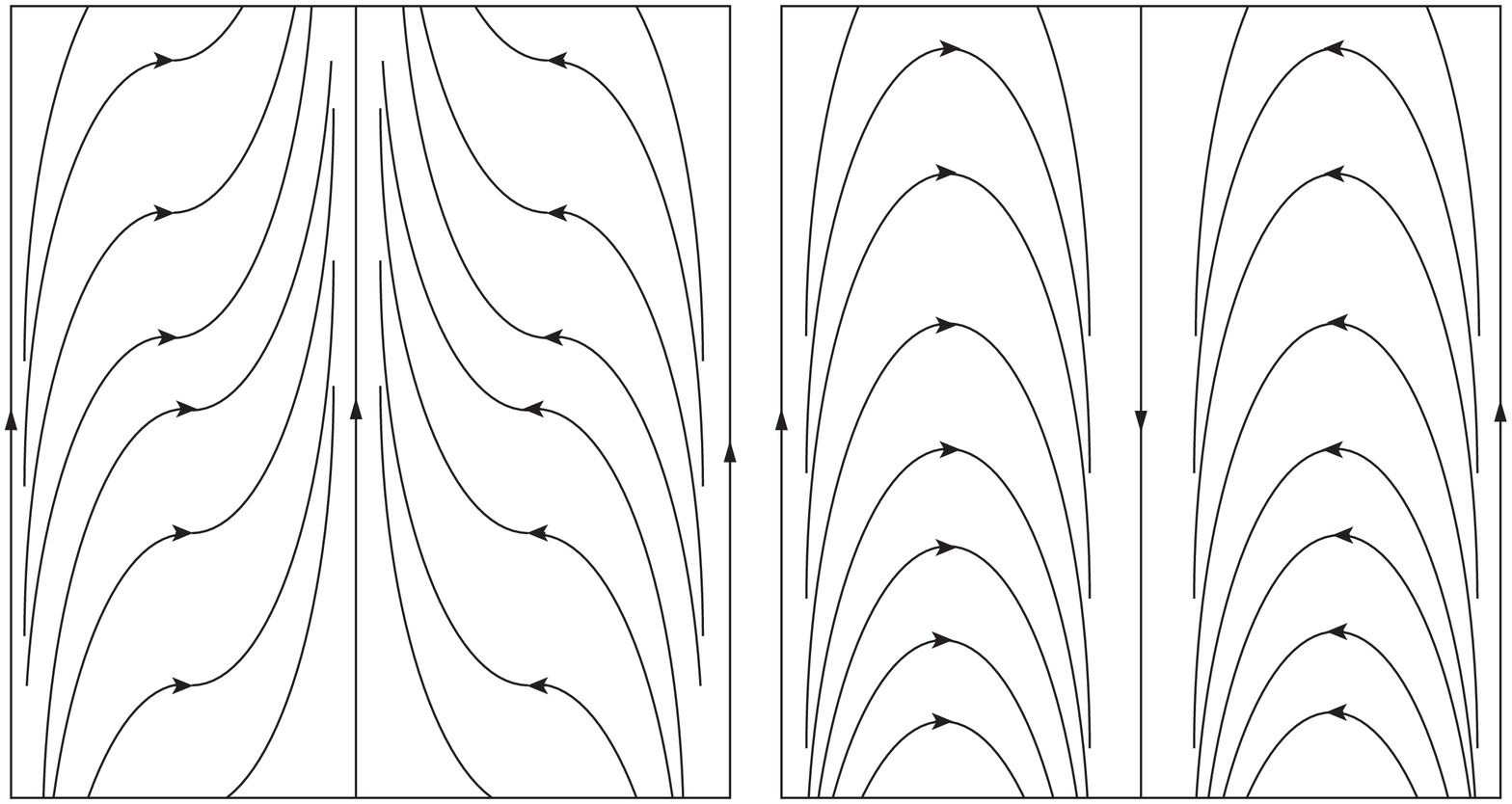}
\end{center} 
\caption{Two flows without singular points 
which are not topologically equivalent 
but which have 
the same extended orbit space.}\label{fig02}
\end{figure}

\begin{figure}
\begin{center}
\includegraphics[scale=0.35]{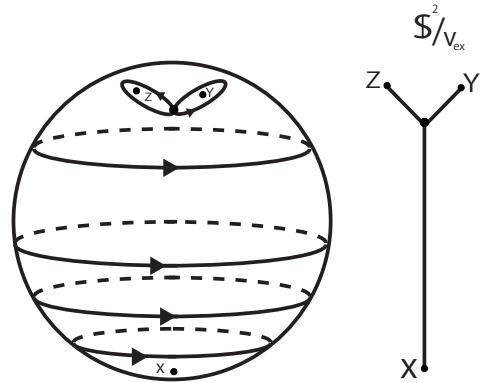}
\end{center} 
\caption{For a flow $v$ with three centers $x, y, z$ 
and the extended orbit space $\S^2/v$, 
replacing a center disk $B_{\sigma}$ ($\sigma = x, y$) into a periodic M\"obius band 
(resp. a periodic orbit which is a boundary component) 
we can obtain two flows $v_{\sigma}$. 
Then the resulting flows $v_x$ and $v_y$ are not topologically equivalent 
but have the same extended orbit space.
}\label{fig03}
\end{figure}

\begin{figure}
\begin{center}
\includegraphics[scale=0.17]{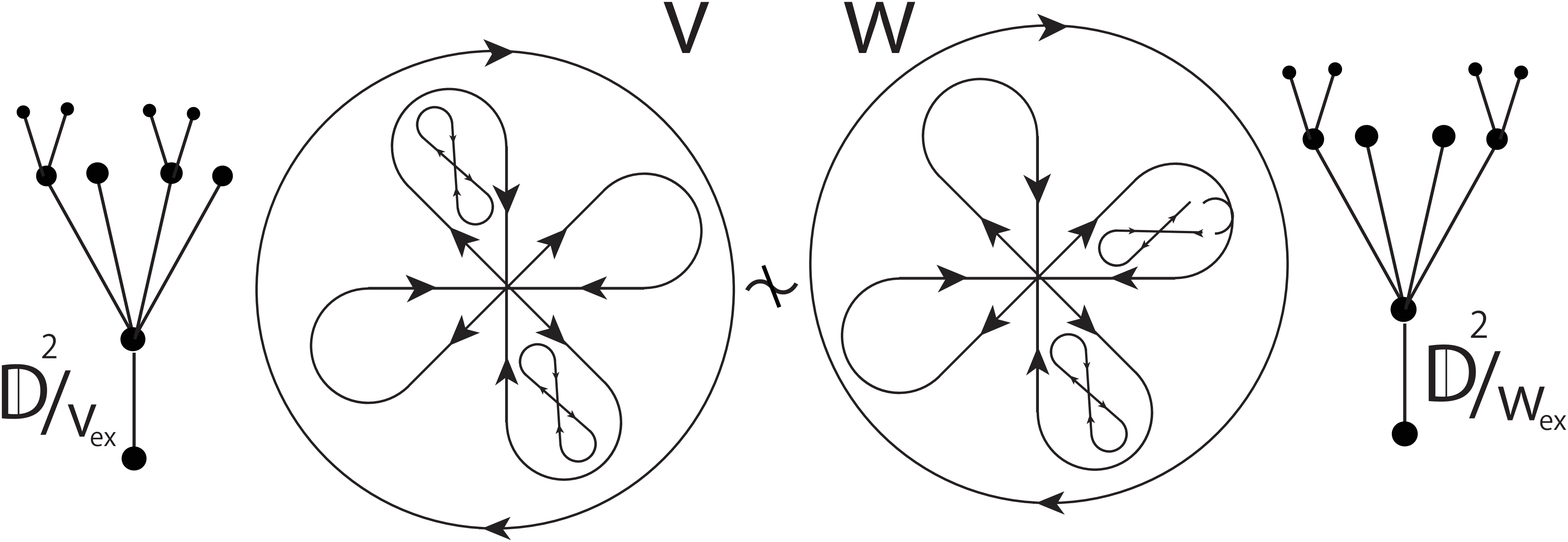}
\end{center} 
\caption{Two flows $v, w$ on a disk $\D^2$ 
which are not topologically equivalent and 
whose multi-saddle connection diagrams are 
not isomorphic as plane graphs (resp. abstract labeled multi-graphs)  
 but isomorphic as abstract multi-graphs
}\label{fig05}
\end{figure}
%



The author wishes to thank the members of the Kyoto Dynamical Systems seminar  
 for useful comments and valuable help.

\bibliography{aomsample}

\begin{thebibliography}{99}
\bibitem{ABZ}
S. Kh. Aranson, G. R. Belitsky and E. V. Zhuzhoma, 
{\it Introduction to the qualitative theory of dynamical systems on surfaces}
Trans. Math. Monographs 153, Amer. Math. Soc., 1996.
\bibitem[BS]{BS}
Bhatia, N. P., Szeg\"o, G. P., 
{\it Stability theory of dynamical systems}
Die Grundlehren der mathematischen Wissenschaften, 
Band 161 Springer-Verlag, New York-Berlin 1970 xi+225 pp. 
\bibitem{C} 
T. M. Cherry,   
{\it Topological properties of solutions of ordinary differential equations} 
Amer. J. Math. 59, 957--982 (1937). 
\bibitem{CG} 
Jack S. Calcut, Robert E. Gompf, 
{\it Orbit spaces of gradient vector fields}
Ergodic Theory Dynam. Systems 33 (2013), no. 6, 1732--1747. 
\bibitem{CGL} 
M.Cobo,  C. Gutierrez, J. Llibre,  
{\it Flows without wandering points on compact connected surfaces}
Trans. Amer. Math. Soc. 362 (2010), no. 9, 4569--4580. 
\bibitem{F} 
J. Franks
{\it Nonsingular Smale Flows on $S^3$}
Topology, 24 (3) (1985) 265--282. 
\bibitem{HM}
W. Hurewicz, K. Menger, 
{\it Dimension and Zusammenhangsstuffe} 
Math. Ann., 100 (1928) pp. 618--633. 
\bibitem{M} 
A. G. Ma\v \i er, 
{\it Trajectories on closed orientable surfaces} 
Mat. Sb. 12 (54) (1943), 71--84
\bibitem{Ma} 
N. Markley, 
{\it On the number of recurrent orbit closures} 
Proc. AMS, 25(1970), no 2, 413--416.
\bibitem{N2}
I. Nikolaev, 
{\it Graphs and flows on surfaces} 
Ergod. Th. Dynam. Syst. 18 (1998), 207--220. 
\bibitem{NZ}
I. Nikolaev and E. Zhuzhoma, 
{\it Flows on 2-Dimensional Manifolds} 
``Lecture Notes in Mathematics'', Vol. 1705, Springer-Verlag, Berlin, 1999. 
\bibitem{N}
I. Nikolaev, 
{\it Non-Wandering Flows on the 2-Manifolds} 
 J. Differential Equations 173 (2001), no. 1, 1--16. 
\bibitem{P}
H. Poincar\'e, 
{\it M\'emoire sur les courbes d\'efinies par une \'equation diff\'erentielle} 
(II), J. de Math. 8 (1882), 251--296.
\bibitem{RF} 
K. A. de Rezende,  R. D. Franzosa, 
{\it  Lyapunov graphs and flows on surfaces} 
Trans. Amer. Math. Soc. {\bf 340} (1993), 767--784. 
\bibitem{T}
L.A. Tumarkin, 
{\it Sur la structure dimensionelle des ensembles ferm\'es} 
C.R. Acad. Paris , 186 (1928) pp. 420--422. 
\bibitem{Y2}
T. Yokoyama,  
{\it Recurrence, pointwise almost periodicity  and orbit closure relation for flows and foliations} 
Topology Appl. 160 (2013) pp.  2196--2206. 
\bibitem{Y}
T. Yokoyama, 
{\it A topological characterization for non-wandering surface flows}, 
Proc. Amer. Math. Soc. 144 (2016), 315--323. 
\end{thebibliography}
\bibliographystyle{aomalpha}

\end{document}